\documentclass[12pt]{article} 
\usepackage[utf8]{inputenc}
\usepackage[english]{babel}
\usepackage[T2A]{fontenc} 

\usepackage[left=2cm, right=2cm, top=2cm, bottom=2cm]{geometry}
\usepackage{amsmath}
\usepackage{mathrsfs} 
\usepackage{amssymb,amsthm}
\usepackage{graphicx, color}
\usepackage{float}
\usepackage{graphicx} 
\graphicspath{ {./images/} }

\newtheorem{df}{Definition}
\newtheorem{theorem}{Theorem}

\newtheorem{hypothesis}{Conjecture}

\begin{document}
\begin{center}
{\bf Quasi-Feynman formulas\\ that provide fast converging Chernoff approximations\\ to solution of parabolic differential equation on the real line\\
Ivan D. Remizov,$^{1,2}$ Alexandr V. Vedenin$^1$}\\ 
ivremizov@yandex.ru, lcsndr@mail.ru
\end{center}

\begin{center}
$^1$HSE University, Nizhny Novgorod, Russia

$^2$IITP RAS, Moscow, Russia
\end{center}

{\bf Abstract.} We construct explicit approximations to the solution of a second-order parabolic partial differential equation on the real line with variable coefficients. The method is based on Chernoff's product formula and uses a new operator-valued function defined through proper Riemann integrals over a bounded interval, which makes the approach readily usable in numerical practice. For sufficiently smooth initial data and coefficients, we prove that the resulting Chernoff approximations converge uniformly in space and time with a quadratic rate, improving the standard first-order estimate. The construction yields a new class of quasi-Feynman formulas that are neither grid-based nor Galerkin-type, but instead rely on semigroup theory and multiple bounded integrals. The theoretical findings are validated by symbolic computation, and the paper contributes both refined error bounds and a practical analytical tool for parabolic problems.

\vskip2mm



\tableofcontents

\section{Introduction}

Theory and practice of solving partial differential equations (PDEs) numerically is an old branch of applied mathematics. Finding solution with appropriate accuracy is extremely important in almost all areas of advanced contemporary engineering, experimental and theoretical physics and chemistry, in particular, in construction of materials, projecting of buildings and devices. The literature on this topic is abundant: almost every large publishing house has in its portfolio a book called Numerical methods for partial differential equations, see \cite{EBY, M2016, book2018, R2016}. All this abundance grows from the fact that it is difficult and important to solve partial differential equations, and there is no numerical method which is applicable to all useful equations. This is why research in this direction is far from ending. 

This paper is devoted to derivation of a formula that expresses arbitrary accurate approximations to solution of PDE with variable coefficients explicitly in terms of these coefficients that play a role of parameters in the problem. This can be considered as a new method for solving parabolic PDEs with variable coefficients.

At the present time we know of a relatively small number of situations where a short formula expresses the solution of a partial differential equation with variable coefficients in terms of these coefficients and initial/boundary conditions. However, with the use of Chernoff's theorem (see original paper by Paul Chernoff \cite{Chernoff}, textbook  \cite{EN2000} or Theorem \ref{Chernth} below) it is possible to write such formulas for approximations to the solution. This fact gives a birth to a new branch of numerical methods which are neither grid-based, nor Galerkin, nor Monte-Carlo, nor iterative, but something in between and at the same time something apart, with the rich flavour of functional analysis. There is a very good recent overview \cite{Butko2020} describing the results and history of the method of Chernoff approximations. 

Employing this method we construct a sequence of functions (Chernoff approximations) which converges (uniformly in $x\in\mathbb{R}$, $t\in[0,t_0]$ for arbitrary $t_0>0$) to the exact solution of the Cauchy problem for diffusion-type second order parabolic equation given by
\begin{equation}\label{firsteq}
\left\{ \begin{array}{ll}
   u'_{t}(t,x)=a(x)u_{xx}''(t,x)+b(x)u'_x(t,x) +c(x)u(t,x)\stackrel{denote}{=}(Hu(t,\cdot))(x),  \\
   u(0,x)=u_0(x).
\end{array} \right.
\end{equation}

Where $x\in\mathbb{R}$, $t\geq0$, $u\colon [0,+\infty)\times \mathbb{R}\to \mathbb{R}$, and functions $a,b,c,u_0\colon\mathbb{R}\to\mathbb{R}$ are bounded and uniformly continuous. We also assume that $a(x)>a_0\equiv\mathrm{const}>0$ for all $x\in\mathbb{R}$, and first and second derivatives of $a,b,c$ are bounded and uniformly continuous as functions $\mathbb{R}\to\mathbb{R}$. 

This paper is dedicated to deriving formulas presented in Theorem \ref{Fform} and Theorem \ref{firstform} that give the solution of (\ref{firsteq}) in terms of $a$, $b$, $c$, $u_0$ assuming that the second order differential operator $H$ is an infinitesimal generator of the $C_0$-semigroup $\left(e^{tH}\right)_{t\geq 0}$ (later we will explain meaning of these words in the Preliminaries section). This assumption is a  standard case in studies of class of evolution equations that the examined equation (\ref{firsteq}) belongs to. According to  general theory of $C_0$-semigroups \cite{EN2000} this assumption implies that the solution of the Cauchy problem (\ref{firsteq}) exists, is bounded and uniformly continuous with respect to $x$ for each $t$, depends on $u_0$ continuously, and can be represented in a form $u(t,x)=\left(e^{tH}u_0\right)(x)$, we will discuss this in more words below in Theorem \ref{semig}. The last equality means that if one knows $e^{tH}$ for each $t>0$ then (\ref{firsteq}) is solved (in the appropriate sense) for each $u_0$. We apply the Chernoff theorem to a specially constructed family of linear bounded operators $(S(t))_{t\geq 0}$ defined as proper Riemann integrals of functions containing $a$, $b$, $c$. This allows us to express $e^{tH}$ in terms of $a$, $b$, $c$, reaching the proposed goal. The aim of the paper do not include detailed discussion of the problem of finding the class of functions in which the solution to (\ref{firsteq}) is unique under certain assumptions on functions $a$, $b$, $c$, $u_0$, but we keep in mind that for the heat equation there are known unbounded solutions. To have high speed of convergence of our approximations we ask additionally that initial condition $u_0$ has bounded uniformly continuous derivatives up to order 8, and coefficients $a,b,c$ have bounded uniformly continuous derivatives up to order 6. 

For (\ref{firsteq}) there are known Chernoff approximations for the solution in terms of shift operators \cite{R-AMC, R-JMP2019} and in terms of improper Riemann integral operators \cite{BGS2010} but they are difficult to use in practice, meanwhile our approximations are better designed for real-world applications.

An important question is how fast the error decreases in the approximation expression $e^{tH}u_0\approx S(t/n)^nu_0$ provided by the Chernoff theorem as $n$ tends to infinity. For a particular case of Schr\"odinger equation this question was discussed in \cite{OSS2012}, but generally this issue is not well developed at the present moment. It was mentioned in \cite{VVGKR2020} that the speed of convergence depends sufficiently on the initial condition, and for some semigroups there are Chernoff functions providing arbitrary slow convergence of approximations. The notion of approximation subspace was proposed in \cite{VVGKR2020} to study this issue, and some numerical experiments were performed recently in \cite{Prud2020, Drag2023SVMO, KNR2026}.

There is a recent improvement \cite{Zag-JFA2020} in direction of research where the convergence $S(t/n)^n\to e^{tH}$ as $n\to \infty$ is studied in operator norm hence convergence $S(t/n)^nu_0\to e^{tH}u_0$ as $n\to \infty$ is uniform with respect to initial condition $u_0$ from unit sphere in $\mathcal{F}$, but examples from \cite{ Prud2020,VVGKR2020, GR1, GR2} show that only convergence $S(t/n)^nu_0\to e^{tH}u_0$ as $n\to \infty$ on each vector $u_0$ is what we should expect in general case.  However, in \cite{GR1, GR2} conditions were found that provide estimate from above on the norm of approximation error that appears in the Chernoff theorem. We will use results of \cite{GR1,GR2} as one of the technical tools, see Theorem \ref{teorApprDU2}.

See also research \cite{Gom2019,GomTom2019,Gom2014} (and references therein) in a more general setting than Chernoff's theorem. The same question arises for the Trotter product formula $e^{A+B}=\lim_{n\to\infty}(e^{A/n}e^{B/n})^n$ which we do not use but which is a particular case of Chernoff's theorem for operator-valued Chernoff function  $G(t)=e^{tA}e^{tB}$, see \cite{ Zag-2, Zag-3,Zag-1}.

The main result of the present paper is Theorem \ref{firstform}.

\section{Preliminaries}

We employ extensively the language and methods of functional analysis and $C_0$-semigroup theory, information on these topics can be found e.g. in classical textbooks \cite{BS-RFA2020,EN2000,HF}. This list is not exhaustive but each of these books is sufficient to understand well all the contents of the paper. Nevertheless, in this section we recall some basic definitions and facts to fix the notation, make the paper self-contained, and highlight the ideas that we will use. All information that is given in Preliminaries section without references is available in \cite{BS-RFA2020,EN2000,HF}, the best choice is probably \cite{EN2000}.

\subsection{$C_0$-semigroups}

Let $\mathcal{F}$ be a real Banach space. Let $\mathscr{L}(\mathcal{F})$ be a set of all bounded linear operators in $\mathcal{F}$. Suppose we have a mapping $V\colon [0,+\infty)\to \mathscr{L}(\mathcal{F}),$ i.e., $V(t)$ is a bounded linear operator $V(t)\colon \mathcal{F}\to \mathcal{F}$ for each $t\geq 0.$ The mapping $V$ is called a \textit{$C_0$-semigroup}, or \textit{a strongly continuous one-parameter semigroup of operators}, if it satisfies the following conditions: 
	
1) $V(0)$ is the identity operator $I$, i.e., for each $\varphi\in \mathcal{F}$ we have $V(0)\varphi=\varphi;$ 
	
2) $V$ maps the addition of numbers in $[0,+\infty)$ into a composition of operators in $\mathscr{L}(\mathcal{F})$, i.e., for all $t\geq 0$ and all $s\geq 0$ we have $V(t+s)=V(t)\circ V(s),$ where $(A\circ B)(\varphi)=A(B(\varphi))=AB\varphi$;
	
3) $V$ is continuous with respect to the strong operator topology in $\mathscr{L}(\mathcal{F})$, i.e., for all $\varphi\in \mathcal{F}$ function $t\longmapsto V(t)\varphi$ is continuous as a mapping from $[0,+\infty)$ to $\mathcal{F}.$

It is known that if $(V(t))_{t\geq 0}$ is a $C_0$-semigroup in Banach space $\mathcal{F}$ then the set $$\left\{\varphi\in \mathcal{F}: \exists \lim_{t\to +0}\frac{V(t)\varphi-\varphi}{t}\right\}\stackrel{denote}{=}D(L)$$ is a dense linear subspace in $\mathcal{F}$. The operator $L$ defined on the domain $D(L)$ by the equality $$L\varphi=\lim_{t\to +0}\frac{V(t)\varphi-\varphi}{t}$$ is called \textit{an infinitesimal generator} (or just \textit{generator} for short) of the $C_0$-semigroup $(V(t))_{t\geq 0}$. The generator $(L,D(L))$ is a closed linear operator that defines the $C_0$-semigroup $(V(t))_{t\geq 0}$ uniquely, and notation $V(t)=e^{tL}$ is used. If $L$ is a bounded operator and $D(L)=\mathcal{F}$, then $e^{tL}$ is indeed the exponent defined by the power series $e^{tL}=\sum_{k=0}^\infty\frac{t^kL^k}{k!}$ converging with respect to the norm topology in $\mathscr{L}(\mathcal{F})$. In most interesting cases the generator is an unbounded differential operator, e.g., Laplacian $\Delta$ is the generator of a multidimensional heat semigroup, and operator $H$ is the generator of semigroup providing solutions to (\ref{firsteq}) in the sense that will be explained below. 

\subsection{Abstract Cauchy problem for evolution equation in Banach space}

One of the reasons for the study of $C_0$-semigroups is their connection with differential equations. We will start from the so-called abstract Cauchy problem (\ref{aCp}) and then explain how it is related to (\ref{firsteq}). Consider Banach space $\mathcal{F}$, vector $u_0\in\mathcal{F}$, linear operator $L\colon D(L)\to\mathcal{F}$ with domain $D(L)\subset\mathcal{F}$, and state the following abstract Cauchy problem: find function $U\colon [0,+\infty)\to\mathcal{F}$ such that 
\begin{equation}\label{aCp}
\left\{ \begin{array}{ll}
U'(t)=LU(t) \ \mathrm{ for }\ t\geq 0,\\
U(0)=u_0.
\end{array} \right.
\end{equation}

Equality $U'(t)=LU(t)$ is called linear evolution equation in $\mathcal{F}$ with operator $L$. Function $U$ is called a classical solution to (\ref{aCp}) iff $U$ has continuous derivative with respect to norm in $\mathcal{F}$, $U(t)\in D(L)$ for all $t\geq 0$ and (\ref{aCp}) holds. Function $U$ is called a mild solution to (\ref{aCp}) iff $U$ is continuous with respect to norm in $\mathcal{F}$, and for all $t\geq 0$ we have $\int\limits_0^tU(s)ds\in D(L)$ and $U(t)=u_0+L\int\limits_0^tU(s)ds$. 

In elementary example provided by  one-dimensional Banach space  $\mathcal{F}=\mathbb{R}$ we have $D(L)=\mathbb{R}$, $u_0\in\mathbb{R}$, and $L$ is a real number which also can be considered as an operator of multiplication by $L$. Then (\ref{aCp}) is a standard ordinary differential equation which is  known to have solution $U(t)=e^{tL}u_0$ both in mild and classical sense which is easy to check by direct computation. It is beautiful and amazing that in case of infinite-dimensional space $\mathcal{F}$ this fact holds true (with some technical details added) and even the formula $U(t)=e^{tL}u_0$ continues to work properly. The following statement gives exact meaning to this emotional comment.

\begin{theorem} \label{semig}
(Proposition 6.2 in \cite{EN2000}, p. 145) 
If $\mathcal{F}$ is a Banach space and $(L,D(L))$ is the generator of $C_0$-semigroup $(e^{tL})_{t\geq 0}$ in $\mathcal{F}$ then:

1. For all $u_0\in D(L)$ there exists unique classical solution to (\ref{aCp}) given by the formula $U(t)=e^{tL}u_0$;

2. For all $u_0\in \mathcal{F}$ there exists unique mild solution to (\ref{aCp}) given by the formula $U(t)=e^{tL}u_0$.
\end{theorem}
Recall also that $e^{tL}$ is a linear bounded operator, so for each $t\geq 0$ solution $U(t)=e^{tL}u_0$ depends on $u_0$ linearly and continuously with respect to norm in $\mathcal{F}$.
Now let us explain how initial-value problems similar to (\ref{firsteq}) and even more general are related with (\ref{aCp}).

\subsection{Cauchy problem for partial differential equation (PDE) of evolution type and its connection to abstract Cauchy problem}

In general situation, if $Q$ is a set (in (\ref{firsteq}) we have $Q=\mathbb{R}$), then the function $u\colon [0,+\infty)\times Q\to \mathbb{R}$, $u\colon (t,x)\longmapsto u(t,x)$ of two variables $(t,x)$ can be considered as a function $U\colon t\longmapsto [x\longmapsto u(t,x)]=u(t,\cdot)$ of one variable $t$
with values in the space of real-valued functions of variable $x$. Let $\mathcal{F}\subset\mathbb{R}^Q$ be a space of (not necessarily all) real-valued functions on $Q$. Suppose that $\mathcal{F}$ is a Banach space with respect to some appropriate norm, and $(L,D(L))$ is a closed densely defined linear operator in $\mathcal{F}$. If $U(t)=u(t,\cdot)\in\mathcal{F}$ for all $t\geq 0$ then one can define $Lu(t,x)=(Lu(t,\cdot))(x)=(LU(t))(x).$ With these definitions it is clear that the Cauchy problem 
\begin{equation}\label{aCpforPDE}
\left\{ \begin{array}{ll}
u'_t(t,x)=Lu(t,x) \ \mathrm{ for }\ t\geq 0, x\in Q\\
u(0,x)=u_0(x)\ \mathrm{ for } \ x\in Q
\end{array} \right.
\end{equation}
is exactly (\ref{aCp}) with the only one difference: derivative $u'_t(t,x)$ in (\ref{aCpforPDE}) above is taken in $\mathbb{R}$ pointwise for each $x\in Q$, meanwhile in (\ref{aCp}) the derivative $U'(t)$ is taken with respect to norm in $\mathcal{F}$: $$u'_t(t,x)=p(t,x)\iff \lim_{\varepsilon\to0}\left|\frac{u(t+\varepsilon,x)-u(t,x)}{\varepsilon}-p(t,x)\right|=0\textrm{ for all }x\in Q,$$ 
$$U'(t)=P(t)\iff \lim_{\varepsilon\to0}\left\|\frac{U(t+\varepsilon)-U(t)}{\varepsilon}-P(t)\right\|=0.$$

If there exists a $C_0$-semigroup $(e^{tL})_{t\geq 0}$ then thanks to theorem \ref{semig} Cauchy problem (\ref{aCp}) has unique classical solution for each $u_0\in D(L)$ and unique mild solution for each $u_0\in\mathcal{F}$. If one selects appropriate definition of solution to (\ref{aCpforPDE}) that meets the above-mentioned difference in understanding the time derivative, then solution to (\ref{aCpforPDE}) is made of solution to (\ref{aCp}), i.e., function $u(t,x)=(U(t))(x)=(e^{tL}u_0)(x)$ solves (\ref{aCpforPDE}) in the appropriate sense. 

Setting (\ref{aCpforPDE}) is very general because the role of $L$ can be played, e.g., by arbitrary differential operator with derivatives of any order and coefficients that depend on $x$ but do not depend on $t$. This covers the case (\ref{firsteq}) in particular. Also $L$ may be a pseudo-differential operator. If $L=-i\mathcal{H}$, where $\mathcal{H}$ is the Hamiltonian of a quantum system then (\ref{aCpforPDE}) is a Schr\"odinger equation. Variable $x$ can run over $\mathbb{R}$, $\mathbb{R}^n,$ manifold, ramified surface, infinite-dimensional space or a subset in these spaces, examples can be found in \cite{Butko2020}.

\subsection{Assumptions and notation that we use}\label{asssect}

Now it is time to introduce the case which we study in the paper. In our consideration $Q=\mathbb{R}$, and $\mathcal{F}=UC_b(\mathbb{R})$ is the space of all bounded and uniformly continuous real-valued functions that are defined on $\mathbb{R}$. We consider $\mathcal{F}=UC_b(\mathbb{R})$ with the uniform norm $\|f\|=\sup_{x\in\mathbb{R}}|f(x)|$. It is easy to check that $UC_b(\mathbb{R})$ is a closed subset of the space $C_b(\mathbb{R})$ of all bounded real-valued functions defined on $\mathbb{R}$. It is a standard fact that $C_b(\mathbb{R})$ is a Banach space with respect to uniform norm, so our $\mathcal{F}=UC_b(\mathbb{R})$ is also a Banach space. We also will use notation
$$
UC_b^n(\mathbb{R})=\{f\in UC_b(\mathbb{R})\;|\;f',f'',\dots,f^{(n)}\in UC_b(\mathbb{R})\}\subset\mathcal{F}=UC_b(\mathbb{R})\textrm{ and } C_b^\infty(\mathbb{R})=\bigcap\limits_{n=1}^\infty UC_b^n(\mathbb{R}).
$$
It is proved in \cite{R-JMP2019} that $C_b^\infty(\mathbb{R})$ is dense in $UC_b(\mathbb{R})$. In our consideration $L=H$, i.e., $$(Lf)(x)=(Hf)(x)=a(x)f''(x)+b(x)f'(x)+c(x)f(x)$$ 
for each point $x\in\mathbb{R}$ and each function $f\in UC_b^\infty(\mathbb{R})$. The domain $D(L)=D(H)$ is defined via the following assumption.

\textbf{Assumption 1.}
We assume that operator $H$ defined above on the set $C_b^\infty(\mathbb{R})$ is closable and its closure $(H,D(H))$ generates a $C_0$-semigroup $(e^{tH})_{t\geq 0}$ in $\mathcal{F}=UC_b(\mathbb{R})$ endowed with the uniform norm $\|f\|=\sup_{x\in\mathbb{R}}|f(x)|$. In other words: $(H,D(H))$ is a closed linear operator that generates a $C_0$-semigroup, and $C_b^\infty(\mathbb{R})$ is a core for $(H,D(H))$.

This assumption is a condition on properties of function $a,b,c$ and usually is not very  restrictive. For example, it holds for the case $a(x)\equiv 1, b(x)\equiv c(x)\equiv 0$ when (\ref{firsteq}) is the heat equation. Also it holds if $a,b,c\in C_b^\infty(\mathbb{R})$ and there exists such $\varepsilon_0>0$ that $a(x)\geq\varepsilon_0$ for all $x\in\mathbb{R}$. We will use Theorem \ref{teorApprDU2} which also states that Assumption 1 holds for the case that we study. 

Domain of the operator which is a closure of some other operator is usually difficult to express explicitly, but at least we know that $C_b^\infty(\mathbb{R})\subset D(L)$. In practice of solving (\ref{firsteq}) the set $D(L)$ is not important because the solution to (\ref{firsteq}) is given by $u(t,x)=(e^{tL}u_0)(x)$, and, independently of how big the set $D(L)$ is, $e^{tL}$ is a linear bounded operator defined everywhere on $\mathcal{F}$. So, with the use of Theorem \ref{semig}, we see that important thing is that $C_0$-semigroup $(e^{tL})_{t\geq 0}$ exists because it gives the solution to (\ref{aCp}) hence the solution to (\ref{firsteq}).

The last thing that we need to discuss here is the time derivative. Suppose that $U(t)$ is the mild or the classical  solution of (\ref{firsteq}) in $\mathcal{F}=UC_b(\mathbb{R})$; what does this mean for $u$ defined as $u(t,x)=(U(t))(x)$? For the mild solution of (\ref{firsteq}) it is required that $U(t)$ is continuous in $t$, which means that $u(t,x)$ is continuous in $t$ uniformly with respect to $x\in\mathbb{R}$. Also $U(t)=u(t,\cdot)\in\mathcal{F}$ which means that $u(t,x)$ is bounded and uniformly continuous in $x$ for each $t\geq 0$. Condition $U(t)=u_0+H\int\limits_0^tU(s)ds$ means that the integral $\int\limits_0^tu(s,x)ds$ exists uniformly with respect to $x\in\mathbb{R}$ and $[x\longmapsto \int\limits_0^tu(s,x)ds]\in D(H)$ where $D(H)$ is defined in the Assumption 1.

Suppose now that $U(t)$ is the classical solution of (\ref{firsteq}). Suppose that we have some function  $p(t)\in\mathcal{F}$, i.e. $p(t,x)$ is bounded and uniformly continuous in $x$ for each $t\geq 0$, and $p(t)=U'(t)$. This means that $0=\lim_{\varepsilon\to0}\left\|\frac{U(t+\varepsilon)-U(t)}{\varepsilon}-p(t)\right\|=\lim_{\varepsilon\to0}\sup_{x\in\mathbb{R}}\left|\frac{u(t+\varepsilon,x)-u(t,x)}{\varepsilon}-p(t,x)\right|$, so $(U'(t))(x)=u'_t(t,x)=p(t,x)$ exists for each fixed $x\in\mathbb{R}$, and, moreover function  $u'_t(t,x)$ is bounded and uniformly continuous in $x$ for each $t\geq 0$. The limit expression above means exactly that $u(t,x)$ is differentiable in $t$ uniformly in $x\in\mathbb{R}$. Finally, it follows from the definition of a classical solution, that $U'(t)$ is continuous with respect to the norm in $\mathcal{F}$, so for each $t\geq 0$ we have $0=\lim_{\varepsilon\to0}\|U'(t+\varepsilon)-U'(t)\|=\lim_{\varepsilon\to0}\sup_{x\in\mathbb{R}}|u'_t(t+\varepsilon,x)-u'_t(t,x)|$ which means exactly that $u_t'(t,x)$ is continuous in $t$ uniformly in $x\in\mathbb{R}$ for each $t\geq 0$. 

Now we have introduced everything which is needed to define the solution of (\ref{firsteq}) which we are eager to find approximately via Chernoff approximations.

\subsection{Setting of the problem that is solved in the paper}
\textbf{Problem 1.}
For the Cauchy problem (\ref{firsteq})
$$
\left\{ \begin{array}{ll}
   u'_{t}(t,x)=a(x)u_{xx}''(t,x)+b(x)u'_x(t,x) +c(x)u(t,x)\stackrel{denoted}{=}(Hu(t,\cdot))(x),  \\
   u(0,x)=u_0(x).
\end{array} \right.
$$
we wish to find (approximately with estimates on the approximation error):

1. Mild solution, i.e. function $u(t,x)$ which for each $t\geq 0$ is: bounded, uniformly continuous in $x$,  continuous in $t$ uniformly with respect to $x\in\mathbb{R}$, and satisfies $u(t,x)=u_0(x)+\left(a(x)\frac{\partial^2}{\partial x^2}+b(x)\frac{\partial}{\partial x}+c(x)\right)\int\limits_0^tu(s,x)ds$ where the integral exists uniformly with respect to $x\in\mathbb{R}$, and $[x\longmapsto \int\limits_0^tu(s,x)ds]\in D(H)$ where $D(H)$ is defined in the Assumption 1. 

2. Classical solution, i.e. function $u(t,x)$ which for each $t\geq 0$ is: bounded, uniformly continuous in $x$, continuous and differentiable in $t$ uniformly with respect to $x\in\mathbb{R}$. Moreover $u_t'(t,x)$ is continuous in $t$ for each $t\geq 0$ uniformly with respect to $x\in\mathbb{R}$, $[x\longmapsto u(t,x)]\in D(H)$ where $D(H)$ is defined in the Assumption 1, and $u'_t(t,x)=\left(a(x)\frac{\partial^2}{\partial x^2}+b(x)\frac{\partial}{\partial x}+c(x)\right)u(t,x)dx$ holds.  

\textbf{Comment on the problem.} Both solutions mentioned in items 1 and 2 above are given by the formula $u(t,x)=(e^{tH}u_0)(x)$ thanks to general theory of $C_0$-semigroups and introductive considerations made above; see propositions 6.3 and 6.4 in the book \cite{EN2000}. So Problem 1 is reduced to problem of approximating of $e^{tH}$ with controlled approximation error. The key instrument for doing this is Theorem \ref{teorApprDU2}. 

\textbf{What do we do to solve this problem.} In terms of linear bounded operators $S(t)$ introduced in Theorem \ref{firstform} below, we will construct a sequence of functions $u_n(t,x)=(S(t/n)^nu_0)(x)$ which for each $t\geq 0$ and each $n=1,2,3,\dots$ are: bounded, uniformly continuous in $x$,  continuous in $t$ uniformly with respect to $x\in\mathbb{R}$. For each $t_0>0$ sequence $u_n(t,x)$ satisfies
$$
\lim_{n\to\infty}\sup_{t\in[0,t_0]}\sup_{x\in\mathbb{R}}|u_n(t,x)-u(t,x)|=0
$$
where $u(t,x)$ is the classical solution of (\ref{firsteq}) for all $u_0\in D(H)$, and $u(t,x)$ is the mild solution of (\ref{firsteq}) for all $u_0\in UC_b(\mathbb{R})$. Moreover, we prove the following estimate on the rate of convergence:
$$\sup_{t\in[0,t_0]}\sup_{x\in\mathbb{R}}|u_n(t,x)-u(t,x)|=O(1/n^2) \textrm{ as } n\to\infty$$ 
at least for all  $u_0\in UC_b^8(\mathbb{R})$ and $a,b,c\in UC_b^6(\mathbb{R})$ satisfying Assumption 1. 


\section{Main result}
\subsection{Approach proposed (the method of Chernoff approximation)}
To reach the aim of the paper we use the method of Chernoff approximation \cite{Butko2020}. Let us explain what it is.

\begin{df}\label{CT} 
(First introduced in \cite{R-JFA2016}) 
Let us say that operator-valued function $G$ is \textit{Chernoff-tangent} to operator $L$ iff the following conditions of Chernoff tangency (CT) hold: 

(CT0). Let $\mathcal{F}$ be a Banach space, and $\mathscr{L}(\mathcal{F})$ be a space of all linear bounded operators in $\mathcal{F}$. Suppose that we have an operator-valued function $G\colon [0, +\infty) \to \mathscr{L}(\mathcal{F})$, or, using other words, we have a family $(G(t))_{t\geq 0}$ of linear bounded operators in $\mathcal{F}$. Closed linear operator $L\colon D(L) \to \mathcal{F}$ is defined on the linear subspace $D(L)\subset\mathcal{F}$ which is dense in  $\mathcal{F}$.

(CT1). Function $G$ is continuous at 0 in the strong operator topology in  $\mathscr{L}(\mathcal{F})$, i.e.  $\lim_{t\to 0}\|G(t)f-G(0)f\|=0$ for each $f \in \mathcal{F}$.

(CT2). $G(0) = I,$ i.~e. $G(0)f=f$ for each $f \in \mathcal{F}$.

(CT3). There exists such a linear subspace $\mathcal{D} \subset \mathcal{F}$ that it is dense in  $\mathcal{F}$ and for each $f \in \mathcal{D}$ there exists a limit $\lim_{t \to 0}(G(t)f-f)/t$; let us denote the value of this limit as   $G'(0)f$.

(CT4). Closure of the operator $(G'(0), \mathcal{D})$ exists and is equal to $(L, D(L))$.
\end{df}

Division of (CT) into (CT0)-(CT1) helps to check these subconditions separately because they are usually proved by methods of different nature. The condition of Chernoff tangency implies that $G(t)f=f+tLf+o(t)$ as $t\to 0$ for all $f$ from the space which is a core for $(L,D(L))$. With the notion of Chernoff tangency the classic Chernoff theorem can be stated in the following form.

\begin{theorem}\label{Chernth}
(\textsc{P.\,R.~Chernoff, 1968}; cf. original paper  \cite{Chernoff}, theorem 5.2 in  \cite{EN2000} and theorem 10.7.21 in \cite{BS-RFA2020}.) 

In the notation of the above definition suppose that $L$ and $G$ satisfy:
	
(E). There exists a $C_0$-semigroup $(e^{tL})_{t\geq 0}$ and its generator is $(L,Dom(L))$.
	
(CT). The function $G$ is Chernoff-tangent to operator $(L,Dom(L))$.

(N). There exists $\omega\in\mathbb{R}$ such that $\|G(t)\|\leq e^{\omega t}$ for all $t\geq 0$.(\ref{CTform})

Then for each $f\in \mathcal{F}$ and each $T>0$ we have 
\begin{equation}\label{CTform}
\lim_{n\to\infty}\sup_{t\in[0,t_0]}\left\|G(t/n)^nf - e^{tL}f\right\|=0,
\end{equation}
where $G(t/n)^n$ is a composition of $n$ copies of linear bounded operator $G(t/n)$.
\end{theorem}

If $G$ is Chernoff-tangent to $L$, then the expression $G(t/n)^nf$ is called a Chernoff approximation expression for $e^{tL}f$, and $G(t/n)^nu_0$ is called a Chernoff formal solution for Cauchy problem $[U'(t)=LU(t); U(0)=u_0]$. If, moreveover, (\ref{CTform}) holds, then $G$ is called a Chernoff function for operator $L$, and $G$ is called Chernoff-equivalent to $C_0$-semigroup $(e^{tL})_{t\geq 0}$; in this case $U(t)=\lim_{n\to\infty}G(t/n)^nu_0=e^{tL}u_0$ is the solution of this Cauchy problem due to Theorem \ref{semig}.

The Chernoff theorem (and definitions derived from it) admit two equivalent wordings: with unbounded time and with arbitrary small time. The first is provided above. The second arises when Chernoff function in (CT) is defined not for all $t\geq 0$, but only for $t\in[0,\delta)$ for fixed small $\delta>0$. The condition $(N)$ in the second case is substituted by the following condition (N'): 

\textit{\textbf{(N')} There exists $\omega_1>0$ such that $\|G(t)\|\leq 1+\omega_1 t$ for all $t\in[0,\delta)$.}

This wording is motivated by the fact that the value of $t/n$ in Chernoff approximation expression $G(t/n)^nf$  becomes arbitrary small as $n\to\infty$ while $t\in[0,t_0]$. This is also in line with the condition (CT3) which itself uses $G(t)$ defined only for small values of $t>0$.

One may ask why finding Chernoff function for operator $L$ is simpler than finding $e^{tL}$ using some other method? Why one should use Chernoff's theorem? The first reason is that there are no standard methods of finding $e^{tL}$ for most interesting operators $L$ with variable coefficients, so usually we can only refer to solving the Cauchy problem (\ref{aCpforPDE}) for each $u_0\in\mathcal{F}$ or to finding the resolvent of $L$ which is not an easy task when $L$ has variable coefficients. The second reason is that Chernoff function $G$ may not have semigroup composition property, i.e. $G(t_1+t_2)$ is allowed to be not  equal to $G(t_1)G(t_2)$), which gives us some freedom in writing the formula for $G$. With this freedom, sometimes it is possible to define $G(t)$ by a short formula containing variable coefficients of operator $L$ as parameters. It is exactly what we do in the paper.

The definition of Chernoff equivalence goes back to 2002's definiton by O.G.~Smo\-lya\-nov \cite{STT}, who since 2000's papers \cite{SWWdan, SWWcan, STmzm} until his death in 2021 systematically applied Chernoff's theorem to solving Cauchy problem for linear evolution equations, see overviews \cite{SmHist, Butko2020}. We will use the following result from \cite{{GR1}}:


\begin{theorem} \label{teorApprDU2}
Suppose that the following three conditions are met:
\begin{enumerate}
\item 
Numbers $m,q\in\{1,2,3,\dots\}$ are  fixed, and $\hat{q}=2\lfloor(q+1)/2\rfloor$.
Functions $a,b,c$ from the class $HC_b^{2m+\hat{q}-2}(\mathbb{R})$ are given, 
such that $\inf_{x\in{\mathbb R}}a(x)>0$. 
Operator $A$ on $UC_b(\mathbb{R})$ with domain $D(A)=HC_b^2(\mathbb{R})$ is defined by the formula
$Au = au''+bu'+cu$.

\item 
Numbers $T>0$, $M\geq 1$ and $\sigma\geq0$ are given. For any $t\in(0,T]$ bounded linear operator $S(t)$ on $UC_b(\mathbb{R})$ is defined
such that $\|S(t)^k\| \le Me^{k\sigma t}$ for any $k=1,2,3,\dots$.

\item 
There exist constant $\alpha\le1$ and nonnegative constants $B_0,B_1,\dots,B_{2m+q}$ such that for all $t\in(0,T]$ and all $f\in UC_b^{2m+q}(\mathbb{R})$ we have 
\begin{equation} \label{eqApprDU2_1}
\bigg\|S(t)f - \sum_{k=0}^{m}\frac{t^kA^kf}{k!}\bigg\| \le t^{m+\alpha}\sum_{i=0}^{2m+q}B_i\|f^{(i)}\|.
\end{equation}
\end{enumerate}

Then the following three statements hold:
\begin{enumerate}
\item 
The closure $\overline{A}$ of operator $A$ in Banach space $UC_b(\mathbb{R})$ is a generator of $C_0$-semigroup  $(e^{t\overline{A}})_{t\ge 0}$ in $UC_b(\mathbb{R})$, and the condition $\|e^{t\overline{A}}\|\le e^{\gamma t}$ for all $t\geq0$ is satisfied, where $\gamma=\sup_{x\in\mathbb{R}}c(x)$.

\item 
Let $w=max(\sigma,\gamma,0)$. Then nonnegative constants $C_0,C_1,\ldots,C_{2m+\hat{q}}$ exists (which are independent of $t$, $T$ and $n$) such that for all $t>0$, all  integer $n\geq n_{\alpha,t}$ (where $n_{\alpha,t}=t/T$ if $\alpha=1$ and $n_{\alpha,t}=max(t/T,t)$ if $\alpha<1$) and all $f\in UC_b^{2m+\hat{q}}(\mathbb{R})$ we have
\begin{equation} \label{eqApprDU2_2}
\big\|S(t/n)^n f - e^{t\overline{A}}f\big\|\leq \frac{Mt^{m+\alpha}e^{wt}}{n^{m-1+\alpha}}\sum_{i=0}^{2m+\hat{q}}C_i\|f^{(i)}\|.
\end{equation}

\item 
If $\alpha>1-m$ then for all $\mathcal{T}>0$ and all $g\in UC_b(\mathbb{R})$ the following equality is true:
\begin{equation} \label{eqApprDU2_3}
\lim_{\mathcal{T}/T\leq n\to\infty}\sup_{t\in(0,\mathcal{T}]}\big\|S(t/n)^ng-e^{t\overline{A}}g\big\| = 0.
\end{equation}
\end{enumerate}
\end{theorem}

\subsection{Derivation of a formula for Chernoff approximations based on S(t)}




Theorem \ref{Fform} constructs a Chernoff function that does not generate rapidly converging approximations. However, it allows the solution to the Cauchy problem (\ref{firsteq}) for a parabolic differential equation to be represented in the form of a new type of Feynman formula. Specifically, in the formula defining the solution, the integral is not taken over the set traversed by the variable $x$, i.e., not along the line $\mathbb{R}$. Instead, the integral is taken over the interval $[-1,1]$, and due to the boundedness of the integrand and the interval of integration, it is a proper Riemann integral. Subsequently, the Chernoff function from Theorem \ref{Fform} is used in Theorem \ref{firstform} to create another Chernoff function (which we again denote by symbol $S$) that generates rapidly converging Chernoff approximations.
\begin{theorem}\label{Fform}
 For all numbers $x\in\mathbb{R}$ and $t \geq 0$, for all $a,b,c\colon\mathbb{R}\to\mathbb{R}$ such that $a, b, c, u_0 \in UC_b(\mathbb {R})\textrm{ and } \inf_{x \in \mathbb{R}} a(x)>0$ we introduce the notation:
\begin{equation}\label{FformEq}
(S(t)f)(x) = \int\limits_{-1}^{1}\frac{1}{2}\Big(1+c(x)t\Big)f\Big(x+(6a(x)t)^{\frac{1}{2}}y+b(x)t\Big)dy.
\end{equation}
For each $\varphi\in C_b^{\infty}(\mathbb {R})$ and each $x\in\mathbb{R}$, we denote
\begin{equation}\label{generator0}
(H\varphi)(x)=a(x)\varphi(x)''+b(x)\varphi(x)'+c(x)\varphi(x).
\end{equation}

Let the space $\mathcal{F}=UC_b(\mathbb {R})$ be equipped with the uniform norm $ \|f\| = sup_{x \in \mathbb {R}}|f(x)|$. Suppose that $C_b^\infty(\mathbb {R})\subset D(H) \subset UC_b(\mathbb {R})$, and the operator $H$ is defined by (\ref{generator0}) for all $\varphi \in C_b^\infty(\mathbb {R})$. Then:

\begin{enumerate}
\item The operator $H$ defined on $HC_b^2(\mathbb {R})=D(H)$ is closable, and the closure $\overline{H}$ of $H$ in the Banach space $UC_b(\mathbb {R})$ is the generator of the $C_0$-semigroup $(e^{tH})_{t\geq 0}$ in $UC_b(\mathbb {R})$, and the condition $\|e^{tH}\|\leq e^{\gamma t}$ holds for all $t\geq 0$, where $\gamma=\sup_{x\in\mathbb {R}}c(x)$.
\item $S(t)$ is a bounded linear operator in $UC_b(\mathbb {R})$ for every $t\geq 0$.
\item $S(t)= I+tH+o(t)$ as $t \to +0$ on $C_b^\infty(\mathbb{R})$, that is, 
$$
\lim_{t\to0}\frac{1}{t}\left\|S(t)f-f-tHf\right\|=0 \textrm{ for each }f\in C_b^\infty(\mathbb{R}).$$
\item $S(t)$ and $(H,D(H))$ satisfy all the conditions of Chernoff's theorem (theorem \ref{Chernth}). In other words, $S$ is the Chernoff function for $H$.
\item For all $a,b,c \in HC_b^2(\mathbb{R})$, $f \in UC_b^4(\mathbb{R})$, $t>0$, $n\geq t$, the inequality
$$
\Bigg\|S(t/n)^nf-e^{tH}f\Bigg\|
\leq e^{wt}\frac{t^{2}}{n}\sum_{j=0}^{4}C_j\|f^{(j)}\|,
$$

where $w$ and $C_0,C_1,\ldots,C_{4}$ are non-negative constants depending on $a,b,c$ and independent of $f$, $t$, and $n$.
\item For all $\mathcal{T}>0$ and all $g\in UC_b(\mathbb{R})$, the following equality holds:
\begin{equation*}
\lim_{\mathcal{T}\leq n\to\infty}\sup_{t\in(0,\mathcal{T}]}\big\|S(t/n)^ng-e^{tH}g\big\| = 0.
\end{equation*}
\end{enumerate}
\end{theorem}

\begin{proof}
1. Let us check the conditions of theorem \ref{teorApprDU2}. For $m=1$, $q=2$ we get $\hat{q}=2\lfloor(q+1)/2\rfloor=2\lfloor(2+1)/2\rfloor=2$, $2m+\hat{q}-2=4$, $2m+\hat{q}=6$. Therefore, we need $a,b,c\in HC_b^2(\mathbb{R})$.

By hypothesis, we have $\inf_{x\in\mathbb{R}}a(x)=a_0>0$.

Functions $a,b,c\in HC_b^{2}(\mathbb{R})$ are given,
such that $\inf_{x\in{\mathbb R}}a(x)>0$.
The operator $H$ on $UC_b(\mathbb{R})$ with domain $D(H)=HC_b^2(\mathbb{R})$ is defined by the formula
$Hu = au''+bu'+cu$ by the conditions of theorem \ref{Fform}.

2. It is known that $C_b(\mathbb{R})$ is a Banach space with respect to the uniform norm. The uniform limit of uniformly continuous bounded functions is a uniformly continuous bounded function, so $UC_b(\mathbb{R})$ is a closed subspace of the complete space $C_b(\mathbb{R})$, hence $UC_b(\mathbb{R})$ itself is a complete normed space, i.e., a Banach space. In (\ref{FformEq}), only bounded continuous functions exist under the integral sign, so the Riemann integral exists, and it is obvious that
$S(t)$ for each $t \geq 0$ is a bounded linear operator on $C_b(\mathbb{R})$ with respect to the uniform norm $\|f\|=\sup_{x\in\mathbb{R}}|f(x)|$. Moreover, the functions under the integral sign are uniformly continuous, so after integrating with respect to $y$, we again obtain uniformly continuous functions due to the properties of elementary Riemann integration. Thus, the space $UC_b(\mathbb{R})$ is invariant with respect to $S(t)$, and therefore $S(t)$ can be viewed as an operator on $UC_b(\mathbb{R})$. Point 2 is proved.

3. We use a Taylor expansion and represent the remainder terms in Lagrangian form. We then group the terms by powers of $t$. Also note that the length of the interval $[-1,1]$ is $2$.
\begin{multline*}
\int\limits_{-1}^{1}\frac{1}{2}(1+c(x)t)f\left(x+(6a(x)t)^{\frac{1}{2}}y+b(x)t\right)dy=(1+c(x)t)f(x)+(1+c(x)t)b(x)f'(x)t+
\\
+(1+c(x)t)\left(a(x)t+\frac{1}{2}(b(x))^2t^2\right)f''(x)+(1+c(x)t)\left((a(x)b(x)t^2+\frac{1}{6}(b(x))^3t^3\right)f'''(x)+
\\
\frac{1}{48}(1+c(x)t)\int\limits_{-1}^{1}\Big(6a(x)t)^{\frac{1}{2}}y+b(x)t\Big)^4f^{(4)}(\xi(t,x,y))dy=
\end{multline*}
\begin{multline*}
=f(x)+t\Big(a(x)f''(x)+b(x)f'(x)+c(x)f(x)\Big)+
\\
+\frac{t^2}{2}\Big(2b(x)c(x)f'(x)+\Big(b(x)^2+2a(x)c(x)\Big)f''(x)+2a(x)b(x)f'''(x)\Big)+
\\
+t^3\Big(\frac{1}{2}b(x)^2c(x)f''(x)+\Big(\frac{1}{6}b(x)^3+a(x)b(x)c(x)\Big)\Big)f'''(x)
+\frac{t^4}{6}b(x)^3c(x)f'''(x)+
\\
+\frac{1}{48}(1+c(x)t)\int\limits_{-1}^{1}\Big(6a(x)t)^{\frac{1}{2}}y+b(x)t\Big)^4f^{(4)}(\xi)dy=f(x)+t(Hf)(x)+o(t)
\end{multline*}

4. Let's check all the conditions of Chernoff's theorem (theorem \ref{Chernth}): (E), (CT), and (N).
Condition (E) is proved in point 1 of Theorem \ref{Fform}. Now let's check (CT0)-(CT4).
(CT0) follows from point 2 of Theorem \ref{Fform}.

Let us prove (CT1), i.e. show that $\lim_{t \to 0}{\|S(t)f-f\|}=0$ for each $f \in UC_b(\mathbb {R})$. Indeed:
\begin{multline*}
\lim_{t \to 0}{\|S(t)f-f\|}
=\lim_{t \to 0}\sup_{x\in\mathbb{R}}\bigg|\int\limits_{-1}^{1}\frac{1}{2}(1+c(x)t)f(x+(6a(x)t)^{\frac{1}{2}}y+b(x)t)dy\bigg|\leq\\
\frac{1}{2}\lim_{t \to 0}\int\limits_{-1}^1\sup_{x\in\mathbb{R}}|f(x+(6a(x)t)^{\frac{1}{2}}y+b(x)t)-2f(x)|dy\ +\\
+\frac{1}{2}\lim_{t \to 0}\left(t\cdot \int\limits_{-1}^1\sup_{x\in\mathbb{R}}|c(x)|\sup_{x\in\mathbb{R}}|f(x+(6a(x)t)^{\frac{1}{2}}y+b(x)t)|dy\right)
\end{multline*}
Let us explain why both limits above are equal to zero. In the first limit, the integrand tends to zero uniformly in $x\in\mathbb{R}$ and $y\in[0,1]$, since the functions $f$ and $s\longmapsto s^{\frac{1}{2}}$ are uniformly continuous, and the functions $a,b$ are bounded. In the second limit, the integrand does not exceed $\|c\|\cdot\|f\|$.

(CT2) Follows from the definition of $S(t)$. Setting $t=0$ in (\ref{FformEq}) yields $(S(0)f)(x)=f(x)$.

To prove (CT3), we choose $\mathcal{D}=C_b^\infty(\mathbb{R})$ 
In \cite{R-JMP2019}, it was proved that $C_b^\infty(\mathbb{R})$ is dense in $UC_b(\mathbb{R})$. We only need to prove that $\lim_{t \to +0}\frac{1}{t}\|S(t)f - f - tHf\|=0$ for all $f\in C_b^\infty(\mathbb{R})$. But this condition is even weaker than the one we have proved in step 2. Therefore, (CT3) is proved.

Condition (CT4) follows from item 1 of Theorem \ref{Fform}. Thus, $S$ is Chernoff tangent to $H$, since all conditions (CT0)-(CT4) are satisfied.

Now we prove condition (N). For an arbitrary $f\in UC_b(\mathbb{R})$ and an arbitrary $t \geq 0$, the following estimates hold:

\begin{multline*}
\|S(t)f\|=\sup_{x\in\mathbb{R}}\bigg|\int\limits_{-1}^{1}\frac{1}{2}(1+c(x)t)f(x+(6a(x)t)^{\frac{1}{2}}y+b(x)t)dy\bigg|\leq \\
\leq \sup_{x\in\mathbb{R}}\left|\int\limits_{-1}^{1}\frac{1}{2}(1+c(x)t)f(x+(6a(x)t)^{\frac{1}{2}}y+b(x)t)dy\right|\leq\\
\leq \frac{1}{2}\int\limits_{-1}^{1}\sup_{x\in\mathbb{R}}|f(x+(6a(x)t)^{\frac{1}{2}}y+b(x)t)|dy+\frac{1}{2}t\int\limits_{-1}^{1}\sup_{x\in\mathbb{R}}|c(x)|\sup_{x\in\mathbb{R}}|f(x+(6a(x)t)^{\frac{1}{2}}y+b(x)t)|dy \leq\\
\leq\frac{1}{2}\|f\|\int\limits_{-1}^{1}dy+\frac{1}{2}t\|c\|\|f\|\int\limits_{-1}^{1}dy=\|f\|\left(1+t\|c\|\right)\leq (1+\eta t)\|f\|.
\end{multline*}

Thus, for the corresponding constant $\eta>0$, we have $\|S(t)\|\leq 1+\eta t$ for all $t\geq 0$, therefore (N) is true. Item 4 is now proven.

5. Apply theorem \ref{teorApprDU2} to $m=1$, $q=2$, $\alpha=1$. According to point 3 of theorem \ref{Fform}, the following decomposition is valid:
\begin{multline*}
\int\limits_{-1}^{1}\frac{1}{2}(1+c(x)t)f\left(x+(6a(x)t)^{\frac{1}{2}}y+b(x)t\right)dy
=f(x)+t(Hf)(x)+
\\
+\frac{t^2}{2}\Big(2b(x)c(x)f'(x)+\Big(b(x)^2+2a(x)c(x)\Big)f''(x)+2a(x)b(x)f'''(x)\Big)+
\\
+t^3\Big(\frac{1}{2}b(x)^2c(x)f''(x)+\Big(\frac{1}{6}b(x)^3+a(x)b(x)c(x)\Big)\Big)f'''(x)
+\frac{t^4}{6}b(x)^3c(x)f'''(x)+
\\
+\frac{1}{48}(1+c(x)t)\int\limits_{-1}^{1}\Big(6a(x)t)^{\frac{1}{2}}y+b(x)t\Big)^4f^{(4)}(\xi)dy=
\end{multline*}

Hence,

\begin{multline*}
\|(S(t)-I-tH)f\|=\sup_{x \in \mathbb{R}}\Bigg|\frac{t^2}{2}\Big(2b(x)c(x)f'(x)+\Big(b(x)^2+2a(x)c(x)\Big)f''(x)+2a(x)b(x)f'''(x)\Big)+\\
+t^3\Big(\frac{1}{2}b(x)^2c(x)f''(x)+\Big(\frac{1}{6}b(x)^3+a(x)b(x)c(x)\Big)\Big)f'''(x)
+t^4\frac{1}{6}b(x)^3c(x)f'''(x)+
\\
+\frac{1}{48}(1+c(x)t)\int\limits_{-1}^{1}\Big(6a(x)t)^{\frac{1}{2}}y+b(x)t\Big)^4f^{(4)}(\xi)dy\Bigg|\leq\\
\leq t^2\Big(\|b\|\|c\|\|f'\|+\frac{1}{2}\|b\|^2\|c\|\|f''\|+\|a\|\|b\|\|f'''\|\Big)+t^3\Big(\frac{1}{2}\|b\|^2\|c\|\|f''\|+\frac{1}{6}\|b\|^3\|f'''\|+\|a\|\|b\|\|c\|\|\|f'''\|\Big)+\\
+\sup_{x \in \mathbb{R}}\Bigg|\frac{1}{48}(1+c(x)t)\int\limits_{-1}^{1}\Big(6a(x)t)^{\frac{1}{2}}y+b(x)t\Big)^4f^{(4)}(\xi)dy\Bigg|
\end{multline*}

Consider
$$
\sup_{x \in \mathbb{R}}\Bigg|\frac{1}{48}(1+c(x)t)\int\limits_{-1}^{1}\Big(6a(x)t)^{\frac{1}{2}}y+b(x)t\Big)^4f^{(4)}(\xi)dy\Bigg|\leq \frac{(1+\|c\|t)}{48}\int\limits_{-1}^{1}\sup_{x \in \mathbb{R}}\Bigg|\Big(6a(x)t)^{\frac{1}{2}}y+b(x)t\Big)^4f^{(4)}(\xi)\Bigg|dy \leq
$$
$$
\leq \frac{(1+\|c\|t)}{24}\Big((6\|a\|t)^{1/2}+\|b\|t\Big)^4\|f^{(4)}\|\leq (P_1t^2+P_2t^{5/2}+P_3t^3+P_4t^{7/2}+P_5t^4+P_6t^{9/2}+P_7t^5)\|f^{(4)}\|
$$
\begin{multline*}
t^2\Big(\|b\|\|c\|\|f'\|+\frac{1}{2}\|b\|^2\|c\|\|f''\|+\|a\|\|b\|\|f'''\|\Big)+t^3\Big(\frac{1}{2}\|b\|^2\|c\|\|f''\|+\frac{1}{6}\|b\|^3\|f'''\|+\|a\|\|b\|\|c\|\|\|f'''\|\Big)\leq\\
\leq Q_1\|f'\|t^2+Q_{2,1}\|f''\|t^2+Q_{2,2}\|f''\|t^3+Q_{3,1}\|f'''\|t^2+Q_{3,1}\|f''\|t^3
\end{multline*}

Let $K_1=Q_1$
$$
K_2=\max(Q_{2,1},Q_{2,2})
$$
$$
K_3=\max(Q_{3,1},Q_{3,2})
$$
$$
K_4=\max(P_1,P_2,P_3,P_4,P_5,P_6,P_7)
$$

We get the following.

$$
\|(S(t)-I-tH)f\| \leq (t^2+t^{5/2}+t^3+t^{7/2}+t^4+t^{9/2}+t^5)\sum_{i=1}^4 K_i\|f^{(i)}\|
$$

$$
Q_j,P_j, K_i \geq 0
$$

For $t \in (0,T]$ and $\alpha \in (0,1]$, the estimate $t^{2+\alpha}=t^2t^{\alpha}\leq t^2T^{\alpha}$ is true. Therefore
$$
\|(S(t)-I-tH)f\| \leq t^2(T^{1/2}+T+T^{3/2}+T^2+T^{5/2}+T^3)\sum_{i=1}^4 K_i\|f^{(i)}\|.
$$
Set $T=1$, then 
$$
B_i=(T^{1/2}+T+T^{3/2}+T^2+T^{5/2}+T^3)K_i=6K_i.
$$
Thus, inequality
$$
\|(S(t)-I-tH)f\| \leq t^2\sum_{i=1}^4 B_i\|f^{(i)}\|
$$
6. Follows from what was proved in point 5 and from theorem \ref{teorApprDU2}.
\end{proof}

The following Theorem \ref{firstform} is the main result of the present paper. It proposes Chernoff function $S(t)$ which provides Chernoff approximations to solutions of (\ref{firsteq}). The solution is given in the form $u(t,x)=\lim_{n\to\infty}u_n(t,x)$ where Chernoff approximations $u_n(t,x)$ are defined by $u_n(t,x)=(S(t/n)^nu_0)(x)$. The operator $S(t)$ that we introduce below is a sum of two integrals, so representation 
$$
u(t,x)=\lim_{n\to\infty}(S(t/n)^nu_0)(x)
$$
is an expression which contains integrals of arbitrary high multiplicity $n$. Such expressions are called quasi-Feynman formulas in \cite{R-JFA2016}, meanwhile representation of a function as a limit of one multiple integral with multiplicity tending to infinity are known as Feynman formulas, see \cite{R-JFA2016, SmHist} for details. Moreover the integrals that appear in Theorem \ref{firstform} are over the cube $[-1,1]^n$ which is a bounded set, so they can be treated as proper Riemann integrals, which is very useful in practice for numerical computation of the integrals (compared with integrals over $\mathbb{R}^n$ that are proposed in \cite{SmHist}). As one can find from the overciew \cite{Butko2020} it is the first concrete example of such representation of the solution of an evolution equation.

Recall the notation that is given in section \ref{asssect}.

\begin{theorem}\label{firstform}
Let $a, b, c \in UC_b^2(\mathbb {R})$ and  $\inf_{x\in\mathbb{R}}a(x)>0$. For all $\varphi\in HC_b^2(\mathbb {R}))$ and all $x\in\mathbb{R}$ let us denote:
$
(H\varphi)(x)=a(x)\varphi''(x)+b(x)\varphi'(x)+c(x)\varphi(x).
$
For all $x\in\mathbb{R}$, $t \geq 0$, $f\in UC_b(\mathbb {R})$ let us denote:
\begin{multline}\label{Chernofffunc}
(S(t)f)(x) = \int\limits_{-1}^{1}\frac{1}{2}\Big(1+c(x)t\Big)f\Big(x+(6a(x)t)^{\frac{1}{2}}y+b(x)t\Big)dy
+\int\limits_{-1}^{1}\Big(\sum_{j=0}^6\beta_j(t,x)y^j\Big)f(x+yt^{\frac{1}{4}})dy,
\end{multline}
where
\begin{multline*}
\beta_0(t,x)=
\frac{1225}{1024}\Big(a(x)c''(x)+b(x)c'(x)+c(x)^2\Big)t^2-\\
-\frac{11025}{512}\Big(a(x)a''(x)+2a(x)b'(x)+b(x)a'(x)\Big)t^{3/2}+\frac{14553}{64}a(x)^2t,
\end{multline*}
$$
\beta_1(t,x)=\frac{3675}{256}\Big(a(x)b''(x)+2a(x)c'(x)+b(x)b'(x)\Big)t^{7/4}-\frac{19845}{32}a(x)a'(x)t^{5/4},
$$
\begin{multline*}
\beta_2(t,x)=
-\frac{11025}{1024}\Big(a(x)c''(x)+b(x)c'(x)+c(x)^2\Big)t^2+\\
+\frac{178605}{512}\Big(a(x)a''(x)+2a(x)b'(x)+b(x)a'(x)\Big)t^{3/2}-\frac{280665}{64}a(x)^2t,
\end{multline*}
$$
\beta_3(t,x)
=-\frac{6615}{128}\Big(a(x)b''(x)+2a(x)c'(x)+b(x)b'(x)\Big)t^{7/4}+\frac{42525}{16}a(x)a'(x)t^{5/4},
$$
\begin{multline*}
\beta_4(t,x)
=\frac{24255}{1024}\Big(a(x)c''(x)+b(x)c'(x)+c(x)^2\Big)t^2-\\
-\frac{467775}{512}\Big(a(x)a''(x)+2a(x)b'(x)+b(x)a'(x)\Big)t^{3/2}+\frac{800415}{64}a(x)^2t.
\end{multline*}
$$
\beta_5(t,x)
=\frac{10395}{256}\Big(a(x)b''(x)+2a(x)c'(x)+b(x)b'(x)\Big)t^{7/4}-\frac{72765}{32}a(x)a'(x)t^{5/4},
$$
\begin{multline*}
\beta_6(t,x)
=-\frac{15015}{1024}\Big(a(x)c''(x)+b(x)c'(x)+c(x)^2\Big)t^2+\\
+\frac{315315}{512}\Big(a(x)a''(x)+2a(x)b'(x)+b(x)a'(x)\Big)t^{3/2}-\frac{567567}{64}a(x)^2t.
\end{multline*}
Then the following statements hold:

1. The operator $H$ defined on $HC_b^2(\mathbb {R})=D(H)$ is closable, its closure is the generator of a $C_0$-semigroup $(e^{tH})_{t\geq0}$ в $UC_b(\mathbb {R})$,  and the condition  $\|e^{tH}\|\leq e^{\gamma t}$  holds for all $t\geq0$, where $\gamma=\sup_{x\in\mathbb{R}}c(x)$.

2. A solution $u(t,x)$ to the Cauchy problem (\ref{firsteq})exists for every $u_0\in UC_b(\mathbb{R})$ as mild solution, and exists as a classical solution for every $u_0\in D(\overline{H})\supset HC_b^2(\mathbb {R})$, in both cases, it is given by the equality $u(t,x)=(e^{tH}u_0)(x)$. 

3. $S(t)f= f+tHf+\frac{1}{2}(tH)^2+o(t^2)$ as $t \to +0$ for each $f\in C_b^\infty(\mathbb{R})$.

4. For all $a,b,c \in HC_b^6(\mathbb{R})$ and each $f \in UC_b^8(\mathbb{R})$, the following holds:
$$
\Bigg\|S(t/n)^nf-e^{tH}f\Bigg\|
\leq e^{w t}\frac{t^{3}}{n^{2}}\sum_{j=0}^{8}C_j\|f^{(j)}\|,
$$
for all  $t>0$
 and all integers $n\geq t$
где $w$ и $C_0,C_1,\ldots,C_{8}$ — are non-negative constants depending on $a,b,c$and independent of $f$, $t$ and $n$.

5. For all $\mathcal{T}>0$ and all $g\in UC_b(\mathbb{R})$ the following equality holds:
\begin{equation*}
\lim_{\mathcal{T}\leq n\to\infty}\sup_{t\in(0,\mathcal{T}]}\big\|S(t/n)^ng-e^{tH}g\big\| = 0.
\end{equation*}
\end{theorem}

\begin{proof}
\hfill \\ 

1. Let's check the conditions of the theorem \ref{firstform}. For $m=2$, $q=3$ we get $\hat{q}=2\lfloor(q+1)/2\rfloor=2\lfloor(3+1)/2\rfloor=4$, $2m+\hat{q}-2=6$, $2m+\hat{q}=8$. Therefore, we need $a,b,c\in HC_b^6(\mathbb{R})$.

By the condition, we have $\inf_{x\in\mathbb{R}}a(x)=a_0>0$.

Given functions $a, b, c\in HC_b^{6}(\mathbb{R})$,
such that $\inf_{x\in{\mathbb{R}}}a(x)>0$.
The operator $H$ on $UC_b(\mathbb{R})$ with domain $D(H)=UC_b^2(\mathbb{R})$ is defined by the formula
$Hu = au''+bu'+cu$ by the conditions of Theorem \ref{firstform}.
Item 1 is proved.

2. Follows from point 1 of theorem \ref{firstform} and point 2 of proposition \ref{semig}. Item 1 is proved.

3. We will not just formally prove item 3, but also show the way that we followed when constructing (\ref{Chernofffunc}) in such a way that item 2 will be true. Probably this is the most interesting part of the paper. The reasoning is long, so we divide it into steps 1-8.

Step 1. We use the Taylor expansion for the first term and represent the residual terms in the Lagrange form. Then we group the terms in powers of t. This is important, because terms with powers of t equal to 3 or higher we can define as $o(t^2)$. Also we note that the length of the segment $[-1,1]$ is 2.
\begin{multline*}
\int\limits_{-1}^{1}\frac{1}{2}(1+c(x)t)f\left(x+(6a(x)t)^{\frac{1}{2}}y+b(x)t\right)dy=(1+c(x)t)f(x)+(1+c(x)t)b(x)f'(x)t+
\\
+(1+c(x)t)\left(a(x)t+\frac{1}{2}(b(x))^2t^2\right)f''(x)+(1+c(x)t)\left((a(x)b(x)t^2+\frac{1}{6}(b(x))^3t^3\right)f'''(x)+
\\
+(1+c(x)t)\left(\frac{3}{10}(a(x))^2t^2+\frac{1}{2}a(x)(b(x))^2t^3+\frac{1}{24}(b(x))^4t^4\right)f^{(4)}(x))+
\\
\frac{1}{240}(1+c(x)t)\int\limits_{-1}^{1}\Big(6a(x)t)^{\frac{1}{2}}y+b(x)t\Big)^5f^{(5)}(\xi_1(t,x,y))dy=
\end{multline*}
\begin{multline*}
=f(x)+t\Big(a(x)f''(x)+b(x)f'(x)+c(x)f(x)\Big)+
\\
+\frac{t^2}{2}\Big(2b(x)c(x)f'(x)+\Big(b(x)^2+2a(x)c(x)\Big)f''(x)+2a(x)b(x)f'''(x)+\frac{3}{5}a(x)^2f^{(4)}(x)\Big)+
\\
+t^3\Big(\frac{1}{2}b(x)^2c(x)f''(x)+\Big(\frac{1}{6}b(x)^3+a(x)b(x)c(x)\Big)f'''(x)+\Big(\frac{1}{2}a(x)b(x)^2+\frac{3}{10}a(x)^2c(x)\Big)f^{(4)}(x)\Big)+
\\
+t^4\Big(\frac{1}{6}b(x)^3c(x)f'''(x)+\Big(\frac{1}{24}b(x)^4+\frac{1}{2}a(x)b(x)^2c(x)\Big)f^{(4)}(x)\Big)+
\\
\frac{t^5}{24}b(x)^4c(x)f^{(4)}(x)+
\\
+\frac{1}{240}(1+c(x)t)\int\limits_{-1}^{1}\Big(6a(x)t)^{\frac{1}{2}}y+b(x)t\Big)^5f^{(5)}(\xi_1)=
\end{multline*}
\begin{multline*}
=f(x)+t(Hf)(x)+
\\
+\frac{t^2}{2}\Big(2b(x)c(x)f'(x)+\Big(b(x)^2+2a(x)c(x)\Big)f''(x)+2a(x)b(x)f'''(x)+\frac{3}{5}a(x)^2f^{(4)}(x)\Big)+
\\
+o(t^2)
\end{multline*}
Recall that $Hf$ is expressed as follows:
$$Hf=af''+bf'+cf$$
$$(Hf)(x)=a(x)f''(x)+b(x)f'(x)+c(x)f(x)$$
Step 2. For convenience, we denote

$\left(2b(x)c(x)f'(x)+\left((b(x))^2+2a(x)c(x)\right)f''(x)+2a(x)b(x)f'''(x)+\frac{3}{5}f^{(4)}(x)\right)$ as $(Mf)(x)$. We get that
$$\int\limits_{-1}^{1}\frac{1}{2}(1+c(x)t)f\left(x+(6a(x)t)^{\frac{1}{2}}y+b(x)t\right)dy = 
f(x)+t(Hf)(x)+\frac{t^2}{2}(Mf)(x)+o(t^2)$$
So we get the condition on $\beta$:
\begin{multline*}
\int\limits_{-1}^{1}(\beta_0(t,x)+\beta_1(t,x)y+\beta_2(t,x)y^2+\beta_3(t,x)y^3+\beta_4(t,x)y^4+\beta_5(t,x)y^5+\beta_6(t,x)y^6)f(x+yt^{\frac{1}{4}})=\\
\frac{t^2}{2}\Big((H^2f)(x)-(Mf)(x)\Big)+o(t^2)
\end{multline*}
because $(S(t)f)(x)$ can be expressed by the following formula:
$$
(S(t)f)(x) = f(x)+t(Hf)(x)+\frac{t^2}{2}(H^2f)(x)+o(t^2)=
$$
$$
=f(x)+t(Hf)(x)+\frac{t^2}{2}(Mf)(x)+\frac{t^2}{2}\Big((H^2f)(x)-(Mf)(x)\Big)+o(t^2)
$$
Step 3. Recall the following formulas for, $Hf$ and $H^2f$
$$Hf=af''+bf'+cf$$
$$H^2f=aaf^{(4)}+(2aa'+2ab)f'''+(aa''+2ab'+2ac+ba'+bb)f''+(ab''+2ac'+bb'+2bc)f'+(ac''+bc'+cc)f$$

Step 4. We will write down $\frac{t^2}{2}\left(H^2f-Mf\right)$ and group its terms by derivatives $f(x)$.
\begin{multline*}
\frac{t^2}{2}(H^2f-Mf)=\frac{t^2}{2}aaf^{(4)}+\frac{t^2}{2}(2aa'+2ab)f'''+\frac{t^2}{2}(aa''+2ab'+2ac+ba'+bb)f''+ 
\\
+\frac{t^2}{2}(ab''+2ac'+bb'+2bc)f'+\frac{t^2}{2}(ac''+bc'+cc)f-
\\
-\frac{t^2}{2}\left(2bcf'+\left(bb+2ac\right)f''(x)+2abf'''+\frac{3}{5}aaf^{(4)}\right)=
\end{multline*}
$$
=\frac{t^2}{2}(ac''+bc'+cc)f+\frac{t^2}{2}(ab''+2ac'+bb')f'+\frac{t^2}{2}(aa''+2ab'+ba')f''+\frac{t^2}{2}(2aa')f'''+ \frac{t^2}{2}\frac{2}{5}aaf^{(4)}
$$
Step 5. We will rewrite $\frac{t^2}{2}((Hf)(x)-(Mf)(x))$ in a new form:
\begin{multline*}
\frac{t^2}{2}((Hf)(x)-(Mf)(x))=\alpha_0(t,x)f(x)+\alpha_1(t,x)f'(x)+\alpha_2(t,x)f''(x)+\alpha_3(t,x)f'''(x)+\\
+\alpha_4(t,x)f^{(4)}(x)+\alpha_5(t,x)f^{(5)}+\alpha_6(t,x)f^{(6)}
\end{multline*}
where
$$
\alpha_0(t,x) = \frac{t^2}{2}(a(x)c''(x)+b(x)c'(x)+c(x)c(x))
$$
$$
\alpha_1(t,x) = \frac{t^2}{2}(a(x)b''(x)+2a(x)c'(x)+b(x)b'(x))
$$
$$
\alpha_2(t,x) = \frac{t^2}{2}(a(x)a''(x)+2a(x)b'(x)+b(x)a'(x))
$$
$$
\alpha_3(t,x) =  \frac{t^2}{2}(2a(x)a'(x))
$$
$$
\alpha_4(t,x) = \frac{t^2}{2}\frac{2}{5}a(x)a(x)
$$
$$
\alpha_5(t,x)=0
$$
$$
\alpha_6(t,x)=0
$$
Step 6. We represent $\frac{t^2}{2}(Hf-Mf)(x)$ as an integral. The subinteral expression is the product of a polynomial in the new variable $y$ and $f (x + yt^{1/4})$:
\begin{multline*}
\frac{t^2}{2}(Hf-Mf)(x)=\alpha_0(t,x)f(x)+\alpha_1(t,x)f'(x)+\alpha_2(t,x)f''(x)+\alpha_3(t,x)f'''(x)+\alpha_4(t,x)f^{(4)}(x)+o(t^2)=\\
=\int\limits_{-1}^{1}\left(\beta_0(t,x)+\beta_1(t,x)y+\beta_2(t,x)y^2+\beta_3(t,x)y^3+\beta_4(t,x)y^4+\beta_5(t,x)y^5+\beta_6(t,x)y^6\right)f(x+y)dy
\end{multline*}

Step 7. We will find the relationship between the coefficients $\alpha$ and $\beta$

decompose $f(x + yt^{1/4})$ by Taylor's formula. 
Thus, we obtain the following system of equations:
$$
\alpha_0 = \int\limits_{-1}^{1}(\beta_0+\beta_1y+\beta_2y^2+\beta_3y^3+\beta_4y^4)dy
$$
$$
\alpha_1 = \int\limits_{-1}^{1}(\beta_0+\beta_1y+\beta_2y^2+\beta_3y^3+\beta_4y^4)yt^{1/4}dy
$$
$$
\alpha_2 = \int\limits_{-1}^{1}\frac{1}{2}(\beta_0+\beta_1y+\beta_2y^2+\beta_3y^3+\beta_4y^4+\beta_5y^5+\beta_6y^6)y^2t^{2/4}dy
$$
$$
\alpha_3 = \int\limits_{-1}^{1}\frac{1}{6}(\beta_0+\beta_1y+\beta_2y^2+\beta_3y^3+\beta_4y^4+\beta_5y^5+\beta_6y^6)y^3t^{3/4}dy
$$
$$
\alpha_4 = \int\limits_{-1}^{1}\frac{1}{24}(\beta_0+\beta_1y+\beta_2y^2+\beta_3y^3+\beta_4y^4+\beta_5y^5+\beta_6y^6)y^4t^{4/4}dy
$$
$$
\alpha_5 = \int\limits_{-1}^{1}\frac{1}{120}(\beta_0+\beta_1y+\beta_2y^2+\beta_3y^3+\beta_4y^4+\beta_5y^5+\beta_6y^6)y^5t^{5/4}dy
$$
$$
\alpha_6 = \int\limits_{-1}^{1}\frac{1}{720}(\beta_0+\beta_1y+\beta_2y^2+\beta_3y^3+\beta_4y^4+\beta_5y^5+\beta_6y^6)y^6t^{6/4}dy
$$
we express the integrals and reduce this system of equations to the linear:

$$
\alpha_0 = 2\beta_0+\frac{2}{3}\beta_2+\frac{2}{5}\beta_4+\frac{2}{7}\beta_6
$$
$$
\alpha_1 = t^{1/4}\left(\frac{2}{3}\beta_1+\frac{2}{5}\beta_3+\frac{2}{7}\beta_5\right)
$$
$$
\alpha_2 = \frac{t^{2/4}}{2}\left(\frac{2}{3}\beta_0+\frac{2}{5}\beta_2+\frac{2}{7}\beta_4+\frac{2}{9}\beta_6\right)
$$
$$
\alpha_3 =  \frac{t^{3/4}}{6}\left(\frac{2}{5}\beta_1+\frac{2}{7}\beta_3+\frac{2}{9}\beta_5\right)
$$
$$
\alpha_4 = \frac{t^{4/4}}{24}\left(\frac{2}{5}\beta_0+\frac{2}{7}\beta_2+\frac{2}{9}\beta_4+\frac{2}{11}\beta_6\right)
$$
$$
\alpha_5 = \frac{t^{5/4}}{120}\left(\frac{2}{7}\beta_1+\frac{2}{9}\beta_3+\frac{2}{11}\beta_5\right)
$$
$$
\alpha_6 = \frac{t^{6/4}}{720}\left(\frac{2}{7}\beta_0+\frac{2}{9}\beta_2+\frac{2}{11}\beta_4+\frac{2}{13}\beta_6\right)
$$
Step 8. Denote $\alpha = (\alpha_0,\alpha_1,\alpha_2,\alpha_3,\alpha_4,\alpha_5, \alpha_6)$, $\beta = (\beta_0,\beta_1,\beta_2,\beta_3,\beta_4, \beta_5, \beta_6)$. 

Let $\alpha$ and $\beta$ be connected by matrix $A$. 
$$
A=
\begin{pmatrix}
\frac{1}{0!}\cdot 2 & 0 & \frac{1}{0!}\cdot\frac{2}{3} & 0 & \frac{1}{0!}\cdot\frac{2}{5} & 0 & \frac{1}{0!}\cdot\frac{2}{7}\\
0 & \frac{t^{1/4}}{1!}\cdot\frac{2}{3} & 0 & \frac{t^{1/4}}{1!}\cdot\frac{2}{5} & 0 & \frac{t^{1/4}}{1!}\cdot\frac{2}{7} & 0\\
\frac{t^{2/4}}{2!}\cdot\frac{2}{3} & 0 & \frac{t^{2/4}}{2!}\cdot\frac{2}{5} & 0 & \frac{t^{2/4}}{2!}\cdot\frac{2}{7}& 0 & \frac{t^{2/4}}{2!}\cdot\frac{2}{9}\\
0 & \frac{t^{3/4}}{3!}\cdot\frac{2}{5} & 0 & \frac{t^{3/4}}{3!}\cdot\frac{2}{7} & 0 & \frac{t^{3/4}}{3!}\cdot\frac{2}{9} & 0\\
\frac{t^{4/4}}{4!}\cdot\frac{2}{5} & 0 & \frac{t^{4/4}}{4!}\cdot\frac{2}{7} & 0 & \frac{t^{4/4}}{4!}\cdot\frac{2}{9} & 0 & \frac{t^{4/4}}{4!}\cdot\frac{2}{11}\\
0 & \frac{t^{5/4}}{5!}\cdot\frac{2}{7} & 0 & \frac{t^{5/4}}{5!}\cdot\frac{2}{9} & 0 & \frac{t^{5/4}}{5!}\cdot\frac{2}{11} & 0\\
\frac{t^{6/4}}{6!}\cdot\frac{2}{7} & 0 & \frac{t^{6/4}}{4!}\cdot\frac{2}{9} & 0 & \frac{t^{6/4}}{4!}\cdot\frac{2}{11} & 0 & \frac{t^{6/4}}{6!}\cdot\frac{2}{13}
\end{pmatrix},\ \ 
$$

\begin{multline*}
\beta_0(t,x)=
\frac{1225}{1024}\Big(a(x)c''(x)+b(x)c'(x)+c(x)^2\Big)t^2-\\
-\frac{11025}{512}\Big(a(x)a''(x)+2a(x)b'(x)+b(x)a'(x)\Big)t^{3/2}+\frac{14553}{64}a(x)^2t,
\end{multline*}
$$
\beta_1(t,x)=\frac{3675}{256}\Big(a(x)b''(x)+2a(x)c'(x)+b(x)b'(x)\Big)t^{7/4}-\frac{19845}{32}a(x)a'(x)t^{5/4},
$$
\begin{multline*}
\beta_2(t,x)=
-\frac{11025}{1024}\Big(a(x)c''(x)+b(x)c'(x)+c(x)^2\Big)t^2+\\
+\frac{178605}{512}\Big(a(x)a''(x)+2a(x)b'(x)+b(x)a'(x)\Big)t^{3/2}-\frac{280665}{64}a(x)^2t,
\end{multline*}
$$
\beta_3(t,x)
=-\frac{6615}{128}\Big(a(x)b''(x)+2a(x)c'(x)+b(x)b'(x)\Big)t^{7/4}+\frac{42525}{16}a(x)a'(x)t^{5/4},
$$
\begin{multline*}
\beta_4(t,x)
=\frac{24255}{1024}\Big(a(x)c''(x)+b(x)c'(x)+c(x)^2\Big)t^2-\\
-\frac{467775}{512}\Big(a(x)a''(x)+2a(x)b'(x)+b(x)a'(x)\Big)t^{3/2}+\frac{800415}{64}a(x)^2t.
\end{multline*}
$$
\beta_5(t,x)
=\frac{10395}{256}\Big(a(x)b''(x)+2a(x)c'(x)+b(x)b'(x)\Big)t^{7/4}-\frac{72765}{32}a(x)a'(x)t^{5/4},
$$
\begin{multline*}
\beta_6(t,x)
=-\frac{15015}{1024}\Big(a(x)c''(x)+b(x)c'(x)+c(x)^2\Big)t^2+\\
+\frac{315315}{512}\Big(a(x)a''(x)+2a(x)b'(x)+b(x)a'(x)\Big)t^{3/2}-\frac{567567}{64}a(x)^2t.
\end{multline*}
Item 3 is proved.

4. We will prove that all the conditions of theorem \ref{teorApprDU2} are satisfied. We proved Condition 1 of theorem \ref{teorApprDU2} in point 1 of theorem \ref{firstform}
Let $T=1$. For any $f\in UC_b(\mathbb{R})$ and any $t\in[0,1]$ the following estimates hold:
\begin{multline*}
\|S(t)f\|=\sup_{x\in\mathbb{R}}\bigg|\int\limits_{-1}^{1}\frac{1}{2}(1+c(x)t)f(x+(6a(x)t)^{\frac{1}{2}}y+b(x)t)dy+\\
+\int\limits_{-1}^{1}(\beta_0(t,x)+\beta_1(t,x)y+\beta_2(t,x)y^2+\beta_3(t,x)y^3+\beta_4(t,x)y^4+\beta_5(t,x)y^5+\beta_6(t,x)y^6)f(x+yt^{\frac{1}{4}})dy\bigg|\leq
\end{multline*}
\begin{multline*}
\leq \sup_{x\in\mathbb{R}}\left|\int\limits_{-1}^{1}\frac{1}{2}(1+c(x)t)f(x+(6a(x)t)^{\frac{1}{2}}y+b(x)t)dy\right|+\\
+\sup_{x\in\mathbb{R}}\left|\int\limits_{-1}^{1}(\beta_0(t,x)+\beta_1(t,x)y+\beta_2(t,x)y^2+\beta_3(t,x)y^3+\beta_4(t,x)y^4+\beta_5(t,x)y^5+\beta_6(t,x)y^6)f(x+yt^{\frac{1}{4}})dy\right|\leq
\end{multline*}
\begin{multline*}
\leq \frac{1}{2}\int\limits_{-1}^{1}\sup_{x\in\mathbb{R}}|f(x+(6a(x)t)^{\frac{1}{2}}y+b(x)t)|dy+\frac{1}{2}t\int\limits_{-1}^{1}\sup_{x\in\mathbb{R}}|c(x)|\sup_{x\in\mathbb{R}}|f(x+(6a(x)t)^{\frac{1}{2}}y+b(x)t)|dy+\\
+\sum_{k=0}^6\sup_{x\in\mathbb{R}}\left|\int\limits_{-1}^{1}\beta_k(t,x)f(x+yt^{\frac{1}{4}})y^kdy\right|\leq\\ 
\leq\frac{1}{2}\|f\|\int\limits_{-1}^{1}dy+\frac{1}{2}t\|c\|\|f\|\int\limits_{-1}^{1}dy+\|f\|\sum_{k=0}^6\|\beta_k(t,\cdot)\|\int\limits_{-1}^{1}|y^k|dy=\\
=\|f\|\left(1+t\|c\|+\sum_{k=0}^4\|\beta_k(t,\cdot)\|\int\limits_{-1}^{1}|y^k|dy\right)\leq (1+\eta t)\|f\|,
\end{multline*}.

Thus, for the corresponding constant $\eta>0$ we have $\|S(t)\|\leq 1+\eta t$ for all $t\in[0,1]$.
Indeed, the condition $t\in[0,1]$ implies $t^\mu\leq t$ for all $\mu\geq 1$. In the definition of $\beta_k$ (see the formulas below formula (\ref{Chernofffunc})) the term $t^\mu$ appears only for $\mu\geq 1$ and is multiplied by bounded functions of the variable $x$ (see the very beginning of the conditions of the theorem \ref{firstform}). This means that there exist constants $\eta_k$ such that $\|\beta_k(t,\cdot)\|\leq \eta_kt$ for all $t\in[0,1]$. Using the estimate $\|S(t)f\|\leq \|f\|\left(1+t\|c\|+\sum_{k=0}^4\|\beta_k(t,\cdot)\|\int\limits_{-1}^{1}|y^k|dy\right)$ proved above, it is easy (although in fact this is not required) to write down a formula expressing $\eta$ in terms of the norms of the functions $a,b,c$ and their derivatives up to the second order.
For a some constant $\eta>0$, $\|S(t)\|\leq 1+\eta t$ holds for all $t\in[0,1]$, so (N') is true. This means that $\|S(t)\|\leq e^{\eta t}$ holds for any $t\in(0,T]$ and $T=1$. From the condition of theorem \ref{teorApprDU2}, consider the numbers $M$ and $\sigma$. Here we get: $M=1$, and $\sigma = \eta$. $\|S(t)^k\| \leq Me^{k\sigma t}$. The second condition  theorem \ref{teorApprDU2} is satisfied.

Suppose $f \in UC_b^8(\mathbb{R})$. We use the Taylor expansion for the first derivative up to the 6th power and the second term up to the 8th power, and represent the remainder terms in Lagrange form. We then group the terms by powers of $t$.

\begin{multline*}
 \int\limits_{-1}^{1}\frac{1}{2}\Big(1+c(x)t\Big)f\Big(x+(6a(x)t)^{\frac{1}{2}}y+b(x)t\Big)dy=f(x)+t(Hf)(x)+\frac{t^2}{2}(Mf)(x)+\\
 +\Big(\frac{1}{6}b(x)^3t^3+a(x)b(x)c(x)t^3+\frac{1}{6}b(x)^3c(x)t^4\Big)f^{(3)}(x)+\\
 +\Big(\frac{1}{2}a(x)b(x)^2t^3+\frac{3}{10}a(x)^2c(x)t^3+\frac{1}{24}b(x)^4t^4+\frac{1}{2}a(x)b(x)^2c(x)t^4+\frac{1}{24}b(x)^4c(x)t^5\Big)f^{(4)}(x)\\
 +\Big(\frac{3}{10}a(x)^2bt^3+\frac{1}{6}a(x)b(x)^3t^4+\frac{3}{10}a(x)^2b(x)c(x)t^4+\frac{1}{120}b(x)^5t^5+\frac{1}{6}a(x)b(x)^3c(x)t^5+\frac{1}{120}b(x)^5c(x)t^6\Big)f^{(5)}(x)+\\
 +\frac{1}{1440}(1+c(x)t)\int\limits_{-1}^{1}\Big(6a(x)t)^{\frac{1}{2}}y+b(x)t\Big)^6f^{(6)}(\xi_1)dy
\end{multline*}
\begin{multline*}
\int\limits_{-1}^{1}\Big(\beta_0(t,x)+\beta_1(t,x)y+\beta_2(t,x)y^2+\beta_3(t,x)y^3+\beta_4(t,x)y^4+\beta_5(t,x)y^5+\beta_6(t,x)y^6\Big)f(x+yt^{\frac{1}{4}})dy=\\
=\frac{t^2}{2}\Big((H^2f)(x)-(Mf)(x)\Big)+\frac{1}{120}\Big(\beta_0(t,x)t^{5/4}\int\limits_{-1}^{1}y^5dy+\beta_1(t,x)t^{5/4}\int\limits_{-1}^{1}y^6dy+\beta_2(t,x)t^{5/4}\int\limits_{-1}^{1}y^7dy+\\
+\beta_3(t,x)t^{5/4}\int\limits_{-1}^{1}y^8dy+\beta_4(t,x)t^{5/4}\int\limits_{-1}^{1}y^9dy+\beta_5(t,x)t^{5/4}\int\limits_{-1}^{1}y^{10}+\beta_6(t,x)t^{5/4}\int\limits_{-1}^{1}y^{11}dydy\Big)f^{(5)}(x)+\\
+\frac{1}{6!}\Big(\beta_0(t,x)t^{6/4}\int\limits_{-1}^{1}y^6dy+\beta_1(t,x)t^{6/4}\int\limits_{-1}^{1}y^7dy+\beta_2(t,x)t^{6/4}\int\limits_{-1}^{1}y^8dy+\\
+\beta_3(t,x)t^{6/4}\int\limits_{-1}^{1}y^9dy+\beta_4(t,x)t^{6/4}\int\limits_{-1}^{1}y^{10}+\beta_5(t,x)t^{6/4}\int\limits_{-1}^{1}y^{11}+\beta_6(t,x)t^{6/4}\int\limits_{-1}^{1}y^{12}dy\Big)f^{(6)}(x)+
\\
+\frac{1}{7!}\Big(\beta_0(t,x)t^{7/4}\int\limits_{-1}^{1}y^7dy+\beta_1(t,x)t^{7/4}\int\limits_{-1}^{1}y^8dy+\beta_2(t,x)t^{7/4}\int\limits_{-1}^{1}y^9dy+\\
+\beta_3(t,x)t^{7/4}\int\limits_{-1}^{1}y^{10}dy+\beta_4(t,x)t^{7/4}\int\limits_{-1}^{1}y^{11}+\beta_5(t,x)t^{7/4}\int\limits_{-1}^{1}y^{12}+\beta_6(t,x)t^{7/4}\int\limits_{-1}^{1}y^{13}dy\Big)f^{(7)}(x)+
\\
+\frac{1}{8!}\int\limits_{-1}^{1}\Big(\beta_0(t,x)+\beta_1(t,x)y+\beta_2(t,x)y^2+\beta_3(t,x)y^3+\beta_4(t,x)y^4+\beta_5(t,x)y^5+\beta_6(t,x)y^6\Big)y^8t^{2}f^{(8)}(\xi_2)dy=
\end{multline*}
\begin{multline*}
=\frac{t^2}{2}\Big((H^2f)(x)-(Mf)(x)\Big)+\frac{1}{7!}\Big(\beta_1(t,x)t^{7/4}\int\limits_{-1}^{1}y^8dy+\beta_3(t,x)t^{7/4}\int\limits_{-1}^{1}y^{10}dy+\beta_5(t,x)t^{7/4}\int\limits_{-1}^{1}y^{12}dy\Big)f^{(7)}(x)+\\
+\frac{1}{8!}\int\limits_{-1}^{1}\Big(\beta_0(t,x)+\beta_1(t,x)y+\beta_2(t,x)y^2+\beta_3(t,x)y^3+\beta_4(t,x)y^4\Big)y^8t^{2}f^{(8)}(\xi_2)dy.
\end{multline*}

We add the first and second terms. Thus, we obtain $(S(t)f)(x)$ in the following form:

\begin{multline*}
(S(t)f)(x) = \int\limits_{-1}^{1}\frac{1}{2}\Big(1+c(x)t\Big)f\Big(x+(6a(x)t)^{\frac{1}{2}}y+b(x)t\Big)dy+
\\
+\int\limits_{-1}^{1}\Big(\beta_0(t,x)+\beta_1(t,x)y+\beta_2(t,x)y^2+\beta_3(t,x)y^3+\beta_4(t,x)y^4\Big)f(x+yt^{\frac{1}{4}})dy = 
\end{multline*}
$$
=f(x)+t(Hf)(x)+\frac{t^2}{2}(H^2f)(x)+\\
+\Big(\frac{1}{6}b(x)^3t^3+a(x)b(x)c(x)t^3+\frac{1}{6}b(x)^3c(x)t^4\Big)f^{(3)}(x)+
$$
\begin{multline*}
 +\Big(\frac{1}{2}a(x)b(x)^2t^3+\frac{3}{10}a(x)^2c(x)t^3+\frac{1}{24}b(x)^4t^4+\frac{1}{2}a(x)b(x)^2c(x)t^4+\frac{1}{24}b(x)^4c(x)t^5\Big)f^{(4)}(x)\\
 +\Big(\frac{3}{10}a(x)^2bt^3+\frac{1}{6}a(x)b(x)^3t^4+\frac{3}{10}a(x)^2b(x)c(x)t^4+\frac{1}{120}b(x)^5t^5+\frac{1}{6}a(x)b(x)^3c(x)t^5+\frac{1}{120}b(x)^5c(x)t^6\Big)f^{(5)}(x)+
 \end{multline*}
 \begin{multline*}
 +\frac{1}{1440}(1+c(x)t)\int\limits_{-1}^{1}\Big(6a(x)t)^{\frac{1}{2}}y+b(x)t\Big)^6f^{(6)}(\xi_1)dy+
\\
+\frac{1}{7!}\Big(\beta_1(t,x)t^{7/4}\int\limits_{-1}^{1}y^8dy+\beta_3(t,x)t^{7/4}\int\limits_{-1}^{1}y^{10}dy+\beta_5(t,x)t^{7/4}\int\limits_{-1}^{1}y^{12}dy\Big)f^{(7)}(x)+
\end{multline*}
$$
+\frac{1}{8!}\int\limits_{-1}^{1}\Big(\beta_0(t,x)+\beta_1(t,x)y+\beta_2(t,x)y^2+\beta_3(t,x)y^3+\beta_4(t,x)y^4+\beta_5(t,x)y^5+\beta_6(t,x)y^6\Big)y^8t^{8/4}f^{(8)}(\xi_2)dy.
$$

Next we see that

\begin{multline*}
\Bigg\|S(t)f-f-tHf-\frac{t^2}{2}H^2f\Bigg\|=\sup_{x\in\mathbb{R}}\Bigg|(S(t)f)(x)-f(x)-t(Hf)(x)-\frac{t^2}{2}(H^2f)(x)\Bigg|=
\end{multline*}
\begin{multline*}
=\sup_{x\in\mathbb{R}}\Bigg|\Big(\frac{1}{6}b(x)^3t^3+a(x)b(x)c(x)t^3+\frac{1}{6}b(x)^3c(x)t^4\Big)f^{(3)}(x)+\\
 +\Big(\frac{1}{2}a(x)b(x)^2t^3+\frac{3}{10}a(x)^2c(x)t^3+\frac{1}{24}b(x)^4t^4+\frac{1}{2}a(x)b(x)^2c(x)t^4+\frac{1}{24}b(x)^4c(x)t^5\Big)f^{(4)}(x)+
\end{multline*}
\begin{multline*}
 +\Big(\frac{3}{10}a(x)^2b(x)t^3+\frac{1}{6}a(x)b(x)^3t^4+\frac{3}{10}a(x)^2b(x)c(x)t^4+\frac{1}{120}b(x)^5t^5+\frac{1}{6}a(x)b(x)^3c(x)t^5+\\
 +\frac{1}{120}b(x)^5c(x)t^6\Big)f^{(5)}(x)+\frac{1}{1440}(1+c(x)t)\int\limits_{-1}^{1}\Big(6a(x)t)^{\frac{1}{2}}y+b(x)t\Big)^6f^{(6)}(\xi_1)dy+
\end{multline*}
\begin{multline*}
+\frac{1}{7!}\Big(\beta_1(t,x)t^{7/4}\int\limits_{-1}^{1}y^8dy+\beta_3(t,x)t^{7/4}\int\limits_{-1}^{1}y^{10}dy+\beta_5(t,x)t^{7/4}\int\limits_{-1}^{1}y^{12}dy\Big)f^{(7)}(x)\\
+\frac{1}{8!}\int\limits_{-1}^{1}\Big(\beta_0(t,x)+\beta_1(t,x)y+\beta_2(t,x)y^2+\beta_3(t,x)y^3+\beta_4(t,x)y^4\Big)y^8t^{2}f^{(8)}(\xi_2)dy\Bigg|\leq
\end{multline*}
\begin{multline*}
\leq\sup_{x\in\mathbb{R}}\Bigg|\Big(\frac{1}{6}b(x)^3t^3+a(x)b(x)c(x)t^3+\frac{1}{6}b(x)^3c(x)t^4\Big)f^{(3)}(x)\Bigg|+\\
 +\sup_{x\in\mathbb{R}}\Bigg|\Big(\frac{1}{2}a(x)b(x)^2t^3+\frac{3}{10}a(x)^2c(x)t^3+\frac{1}{24}b(x)^4t^4+\frac{1}{2}a(x)b(x)^2c(x)t^4+\frac{1}{24}b(x)^4c(x)t^5\Big)f^{(4)}(x)\Bigg|+
\end{multline*}
\begin{multline*}
 +\sup_{x\in\mathbb{R}}\Bigg|\Big(\frac{3}{10}a(x)^2b(x)t^3+\frac{1}{6}a(x)b(x)^3t^4+\frac{3}{10}a(x)^2b(x)c(x)t^4+\frac{1}{120}b(x)^5t^5+\frac{1}{6}a(x)b(x)^3c(x)t^5+\\
 +\frac{1}{120}b(x)^5c(x)t^6\Big)f^{(5)}(x)\Bigg|+\sup_{x\in\mathbb{R}}\Bigg|\frac{1}{1440}(1+c(x)t)\int\limits_{-1}^{1}\Big(6a(x)t)^{\frac{1}{2}}y+b(x)t\Big)^6f^{(6)}(\xi_1)dy\Bigg|+
\end{multline*}
\begin{multline*}
 +\sup_{x\in\mathbb{R}}\Big|\frac{1}{7!}\Big(\beta_1(t,x)t^{7/4}\int\limits_{-1}^{1}y^8dy+\beta_3(t,x)t^{7/4}\int\limits_{-1}^{1}y^{10}dy+\beta_5(t,x)t^{7/4}\int\limits_{-1}^{1}y^{12}dy\Big)f^{(7)}(x)\Big|+\\
+\sup_{x\in\mathbb{R}}\Bigg|\frac{1}{8!}\int\limits_{-1}^{1}\Big(\beta_0(t,x)+\beta_1(t,x)y+\beta_2(t,x)y^2+\beta_3(t,x)y^3+\beta_4(t,x)y^4\Big)y^8t^{2}f^{(8)}(\xi_2)dy\Bigg|
\end{multline*}

We consider each term and evaluate them in terms of the norms of their derivatives. The functions $a$, $b$, $c$, $a'$, $b'$, $c'$, $a''$, $b''$, and $c''$ are bounded, so their sum and product are also bounded functions. We evaluate them in terms of the constants $K_{i,j,k}$, where $i$ is the number of terms, $j$ is the number of powers of $t$, and $k$ is the number of derivatives.

1. 
$$
\sup_{x\in\mathbb{R}}\Bigg|\Big(\frac{1}{6}b(x)^3t^3+a(x)b(x)c(x)t^3+\frac{1}{6}b(x)^3c(x)t^4\Big)f^{(3)}(x)\Bigg| \leq
K_{1,3,3}t^3\|f^{(3)}\|+K_{1,4,3}t^4\|f^{(3)}\|
$$

2.
\begin{multline*}
\sup_{x\in\mathbb{R}}\Bigg|\Big(\frac{1}{2}a(x)b(x)^2t^3+\frac{3}{10}a(x)^2c(x)t^3+\frac{1}{24}b(x)^4t^4+\frac{1}{2}a(x)b(x)^2c(x)t^4+\frac{1}{24}b(x)^4c(x)t^5\Big)f^{(4)}(x)\Bigg|\leq\\
\leq K_{2,3,4}t^3\|f^{(4)}\|+K_{2,4,4}t^4\|f^{(4)}\|+K_{2,5,4}t^5\|f^{(4)}\|
\end{multline*}

3.
\begin{multline*}
\sup_{x\in\mathbb{R}}\Bigg|\Big(\frac{3}{10}a(x)^2b(x)t^3+\frac{1}{6}a(x)b(x)^3t^4+\frac{3}{10}a(x)^2b(x)c(x)t^4+\frac{1}{120}b(x)^5t^5+\frac{1}{6}a(x)b(x)^3c(x)t^5+\\
 +\frac{1}{120}b(x)^5c(x)t^6\Big)f^{(5)}(x)\Bigg|\leq K_{3,3,5}t^3\|f^{(5)}\|+K_{3,4,5}t^4\|f^{(5)}\|+K_{3,5,5}t^5\|f^{(5)}\|+K_{3,6,5}t^6\|f^{(5)}\|
\end{multline*}

4.
\begin{multline*}
\sup_{x\in\mathbb{R}}\Bigg|\frac{1}{1440}(1+c(x)t)\int\limits_{-1}^{1}\Big(6a(x)t)^{\frac{1}{2}}y+b(x)t\Big)^6f^{(6)}(\xi_1)dy\Bigg|\leq\frac{1}{1440}(1+\|c\|t)\cdot 2\Big(6\|a\|^{\frac{1}{2}}t^{\frac{1}{2}}+\|b\|t\Big)^6\|f^{(6)}\|\leq\\
\leq K_{4,3,6}t^3\|f^{(6)}\|+K_{4,7/2,6}t^{7/2}\|f^{(6)}\|+K_{4,4,6}t^4\|f^{(6)}\|+K_{4,9/2,6}t^{9/2}\|f^{(6)}\|+K_{4,5,6}t^5\|f^{(6)}\|+\\
+K_{4,11/2,6}t^{11/2}\|f^{(6)}\|+K_{4,6,6}t^{6}\|f^{(6)}\|+K_{4,13/2,6}t^{13/2}\|f^{(6)}\|+K_{4,7,6}t^{7}\|f^{(6)}\|
\end{multline*}

5. 
\begin{multline*}
\sup_{x\in\mathbb{R}}\Big|\frac{1}{7!}\Big(\beta_1(t,x)t^{7/4}\int\limits_{-1}^{1}y^8dy+\beta_3(t,x)t^{7/4}\int\limits_{-1}^{1}y^{10}dy+\beta_5(t,x)t^{7/4}\int\limits_{-1}^{1}y^{12}dy\Big)f^{(7)}(x)\Big|
\leq\\ 
\leq K_{5,3,7}t^{3}\|f^{(7)}\|+K_{5,7/2,7}t^{7/2}\|f^{(7)}\|
\end{multline*}

6.
\begin{multline*}
\sup_{x\in\mathbb{R}}\Bigg|\frac{1}{8!}\int\limits_{-1}^{1}\Big(\beta_0(t,x)+\beta_1(t,x)y+\beta_2(t,x)y^2+\beta_3(t,x)y^3+\beta_4(t,x)y^4\Big)y^8t^{2}f^{(8)}(\xi_2)dy\Bigg|\leq\\
\leq \frac{2}{8!}\sup_{x\in\mathbb{R}}|\beta_0(t,x)|t^{2}\|f^{(8)}\|+\frac{2}{8!}\sup_{x\in\mathbb{R}}|\beta_1(t,x)|t^{2}\|f^{(8)}\|+\frac{2}{8!}\sup_{x\in\mathbb{R}}|\beta_2(t,x)|t^{2}\|f^{(8)}\|+\\
+\frac{2}{8!}\sup_{x\in\mathbb{R}}|\beta_3(t,x)|t^{2}\|f^{8}\|+\frac{2}{8!}\sup_{x\in\mathbb{R}}|\beta_4(t,x)|t^{2}\|f^{(8)}\|\leq
\end{multline*}
\begin{multline*}
\leq K_{6,3,8}t^{3}\|f^{(8)}\|+K_{6,13/4,8}t^{13/4}\|f^{(8)}\|+K_{6,7/2,8}t^{7/2}\|f^{(8)}\|+K_{6,15/4,8}t^{15/4}\|f^{(8)}\|+K_{6,4,8}t^{4}\|f^{(8)}\|
\end{multline*}

Then the following is true:

\begin{multline*}
\Bigg\|S(t)f-f-tHf-\frac{t^2}{2}H^2f\Bigg\|\leq \\
\leq \Big(t^3+t^{13/4}+t^{7/2}+t^{15/4}+t^{4}+t^{9/2}+t^{5}+t^{11/2}+t^{6}+t^{13/2}+t^{7}\Big)\sum_{j=0}^{8}K_j\|f^{(j)}\|.
\end{multline*}

where $K_0=K_1$=$K_2=0$,
$K_3=\max(K_{1,3,3},K_{1,4,3})$,
$K_4=\max(K_{2,3,4},K_{2,4,4},K_{2,5,4})$,

$K_5=\max(K_{3,3,5},K_{3,4,5},K_{3,5,5},K_{3,6,5}$,

$K_6=\max(K_{4,3,6},K_{4,7/2,6},K_{4,4,6},K_{4,9/2,6},K_{4,5,6},K_{4,11/2,6},K_{4,6,6},K_{4,13/2,6},K_{4,7,6}$,

$K_7=\max(K_{7,3,7},K_{7,7/2,7})$,
$K_8=\max(K_{8,3,8},K_{8,13/4,8},K_{8,7/2,8},K_{8,15/4,8},K_{8,4,8})$

$K_0,K_1,\ldots,K_{8}$ are non-negative constants independent of $t$.
\\
Consider $t\in(0,1]$.

In this case, the following inequality holds:

$$
t^3+t^{13/4}+t^{7/2}+t^{15/4}+t^{4}+t^{9/2}+t^{5}+t^{11/2}+t^{6}+t^{13/2}+t^{7} \leq 11t^3
$$

We denote $B_j=11K_j$, then

$$
\Bigg\|S(t)f-f-tHf-\frac{t^2}{2}H^2f\Bigg\|\leq t^3\sum_{j=0}^{8}B_j\|f^{(j)}\|.
$$
This is true for every $t\in(0,T]$, $T=1$. Moreover,
$w=\max(\sigma,\gamma,0)$, where $\gamma=\sup_{x\in\mathbb{R}}c(x)$. We obtain that $\sigma=\eta$. We obtain: $w=\max(\eta,\sup_{x\in\mathbb{R}}c(x),0)$. Condition 3 theorem \ref{teorApprDU2} is satisfied.
Item 4 is proved.

5. It follows from the proof of item 4. All conditions of theorem \ref{teorApprDU2} are satisfied.

Theorem \ref{firstform} is proven.

\end{proof}





\section*{Acknowledgements}

Authors are thankful to O.E.~Galkin for fruitful discussion of the issues raised in the paper. Theorem 4 was created in the Dobrushin's mathematical laboratory IITP RAS. Theorem 5 was created in the HSE University Nizhny Novgorod.

\newpage
\section{Appendix: symbolic computation verification of the formula for $S(t)$}

As can be seen in the text above, many concrete number coefficients were obtained via rather lengthy calculations, which potentially leaves space for a mistake to occur. In order to prove that the formula presented in (\ref{Chernofffunc}) is correct and indeed has the expansion in $t$ as we state in item 3 of Theorem \ref{firstform}, we wrote a program on Wolfram Language. We aim to check that 
$$S(t)f= f+tHf+\frac{1}{2}(tH)^2+o(t^2)$$ 
as $t \to +0$ for each $f\in C_b^\infty(\mathbb{R})$. We show the code (input) and the result (output), so everyone can check and make sure that we found the correct number values of the coefficients. We present several screenshots and comments to them.

For brevity, we omit argument $x$ below in functions $a,b,c$ and their derivatives, i.e. we write $a,a',a'',b,b',b'',c,c',c''$ instead of $a(x),a'(x),a''(x),b(x),b'(x),b''(x),c(x),c'(x),c''(x)$. 
This omitting is justifiable because, first, only derivatives of function $f$ are calculated symbolically, and second, $x$ is constant when we expand $f$ via Taylor's formula (assuming $t\to 0$) and integrate in $y$. However, we keep $x$ in argument of function $f$ and it is not possible to omit it. 
\vskip 2cm
There will be some white spaces below because we wish our screenshots be of appropriate (readable) size and followed by the comments on the same page. 
First, we write $\beta_0$, ..., $\beta_6$, see Wolfram input and output below:

\begin{figure}[H]
\includegraphics[width=18cm, height=10cm]{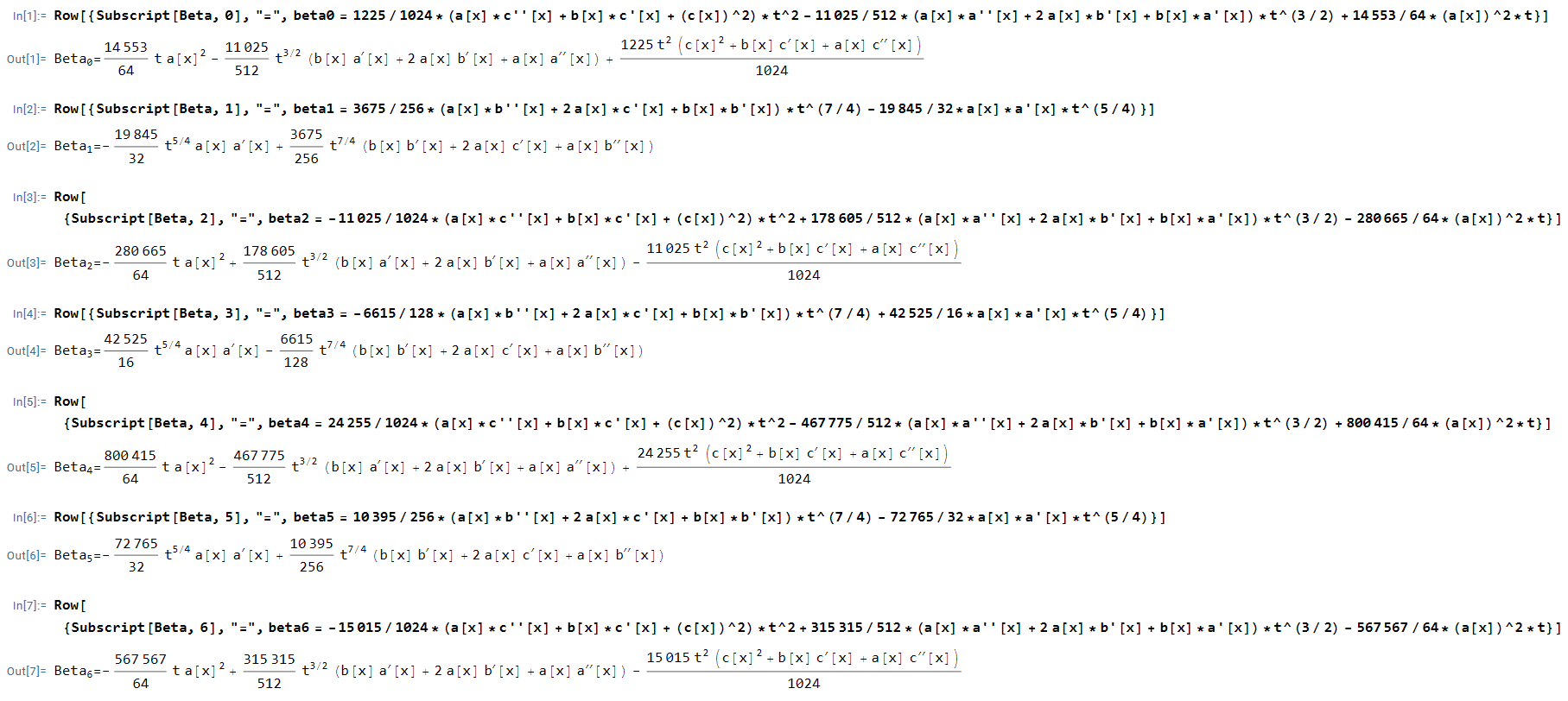}
\caption{Verification of the Chernoff function from Theorem 5 using Wolfram Mathematica 13.3. Part 1/5.}
\label{th5pic1}
\end{figure}

\newpage
We use symbol $Theorem_5\_Taylor_1$ to denote the Taylor expansion of the function\\ $x\mapsto f(x+(6a(x)t)^{1/2}y+b(x)t)$  at the point $x$ up to the 5th derivative, assuming that number  $y\in[-1,1]$ is fixed and $t\to 0$, which gives us that $(6a(x)t)^{1/2}y+b(x)t\to0$. The residual terms with derivatives with powers greater than 5 we can represent in the Lagrange's form for $f \in C_b^6(\mathbb{R})$. 

We use symbol $Theorem_5\_Taylor_2$ to denote the Taylor expansion of the function
$x\mapsto f(x+yt^{1/4})$  at the point $x$ up to the 7th derivative, assuming that number  $y\in[-1,1]$ is fixed and $t\to 0$, which gives us that $yt^{1/4})$. The residual terms with derivatives with powers greater than 7 we can represent in the Lagrange's form for $f \in C_b^8(\mathbb{R})$.


Then, we integrate $\int\limits_{-1}^{1}\frac{1}{2}\Big(1+c(x)t\Big)Theorem_5\_Taylor_1dy+\int\limits_{-1}^{1}\Big(\beta_0(t,x)+\beta_1(t,x)y+\beta_2(t,x)y^2+\beta_3(t,x)y^3+\beta_4(t,x)y^4+\beta_5(t,x)y^5+\beta_6(t,x)y^6\Big)Theorem_5\_Taylor_2dy$ and denote this integral as $Theorem_5\_IntTaylorSum$.

We group the terms of this expression into powers of $t$ and denote $Theorem_5\_Sum$

\begin{figure}[H]
\includegraphics[width=18cm, height=12cm]{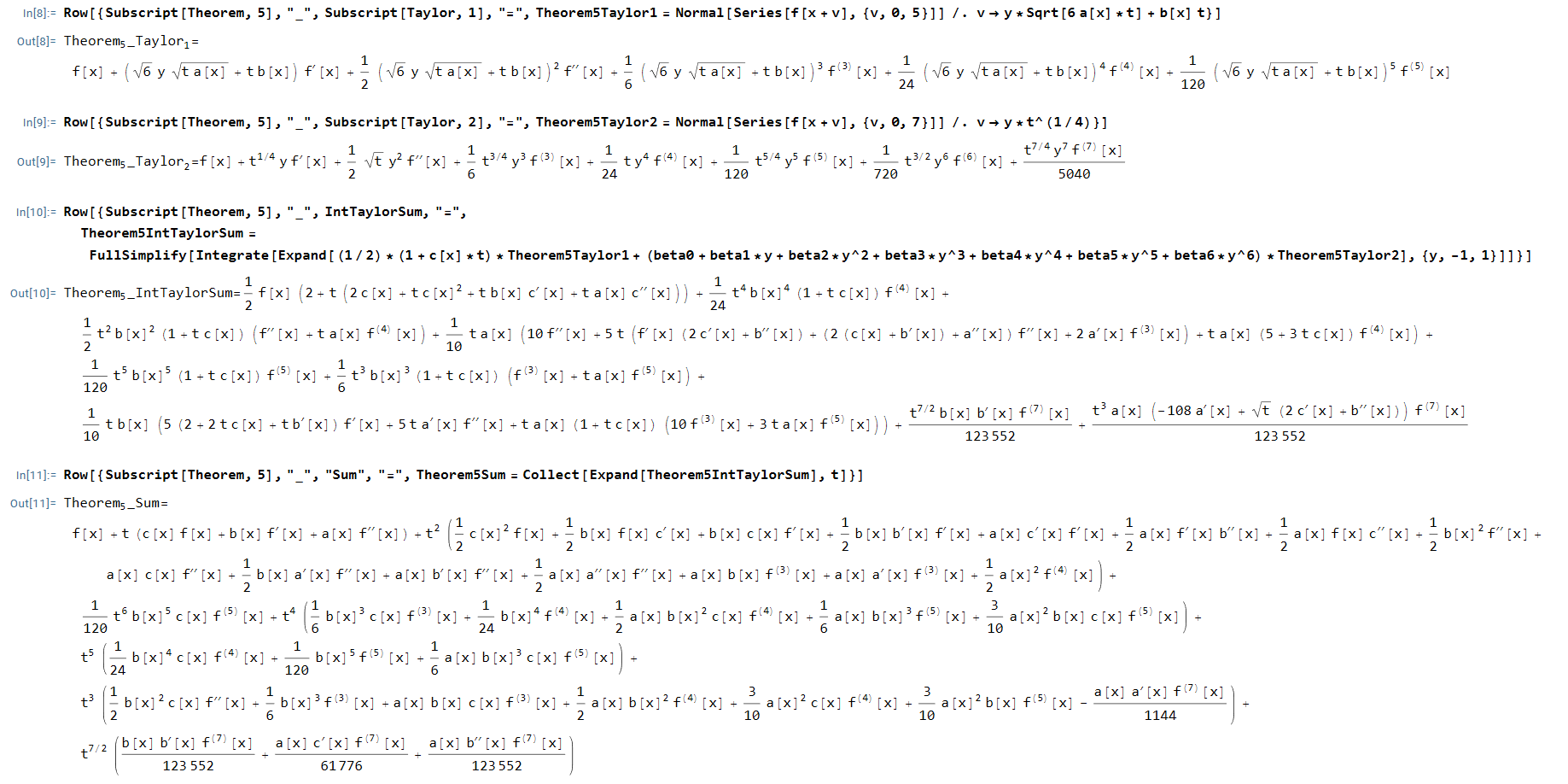}
\caption{Verification of the Chernoff function from Theorem 5 using Wolfram Mathematica 13.3. Part 2/5.}
\label{th5pic2}
\end{figure}

\newpage


We also consider The Taylor expansion of the function $f(x+(6a(x)t)^{1/2}y+b(x)t)$  at the point $x$.
We represent the remainder with the 6th derivative as a Lagrange's form and denote it $Teorem_5\_Lagrange$
Then we write the integral $\int\limits_{-1}^{1}\frac{1}{2}\Big(1+c(x)t\Big)Theorem_5\_Lagrange dy$
We denote this integral $Theorem_5\_IntLagrange$.

\begin{figure}[H]
\includegraphics[width=18cm, height=4cm]{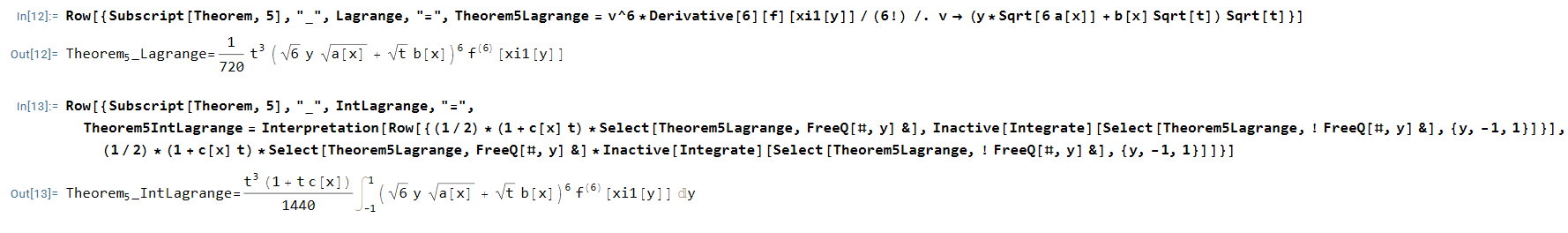}
\caption{Verification of the Chernoff function from Theorem 5 using Wolfram Mathematica 13.3. Part 3/5.}
\label{th5pic3}
\end{figure}

\newpage
We also consider The Taylor expansion of the function $f(x+yt^{1/4})$  at the point $x$.
We represent the remainder with the 8th derivative as a Lagrange's form, multiply by $\Big(\beta_0(t,x)+\beta_1(t,x)y+\beta_2(t,x)y^2+\beta_3(t,x)y^3+\beta_4(t,x)y^4+\beta_5(t,x)y^5+\beta_6(t,x)y^6\Big)$
 and denote it $Teorem_5\_BetaLagrange$
Then we write the integral $\int\limits_{-1}^{1}\frac{1}{2}\Big(1+c(x)t\Big)Theorem_5\_Lagrange dy$
We denote this integral $\int\limits_{-1}^{1}\frac{1}{2}\Big(1+c(x)t\Big)Theorem_5\_BetaLagrange dy$
We denote this integral 

$Theorem_5\_IntLagrange$.$Theorem_5\_IntBetaLagrange$.
\begin{figure}[H]
\includegraphics[width=18cm, height=9cm]{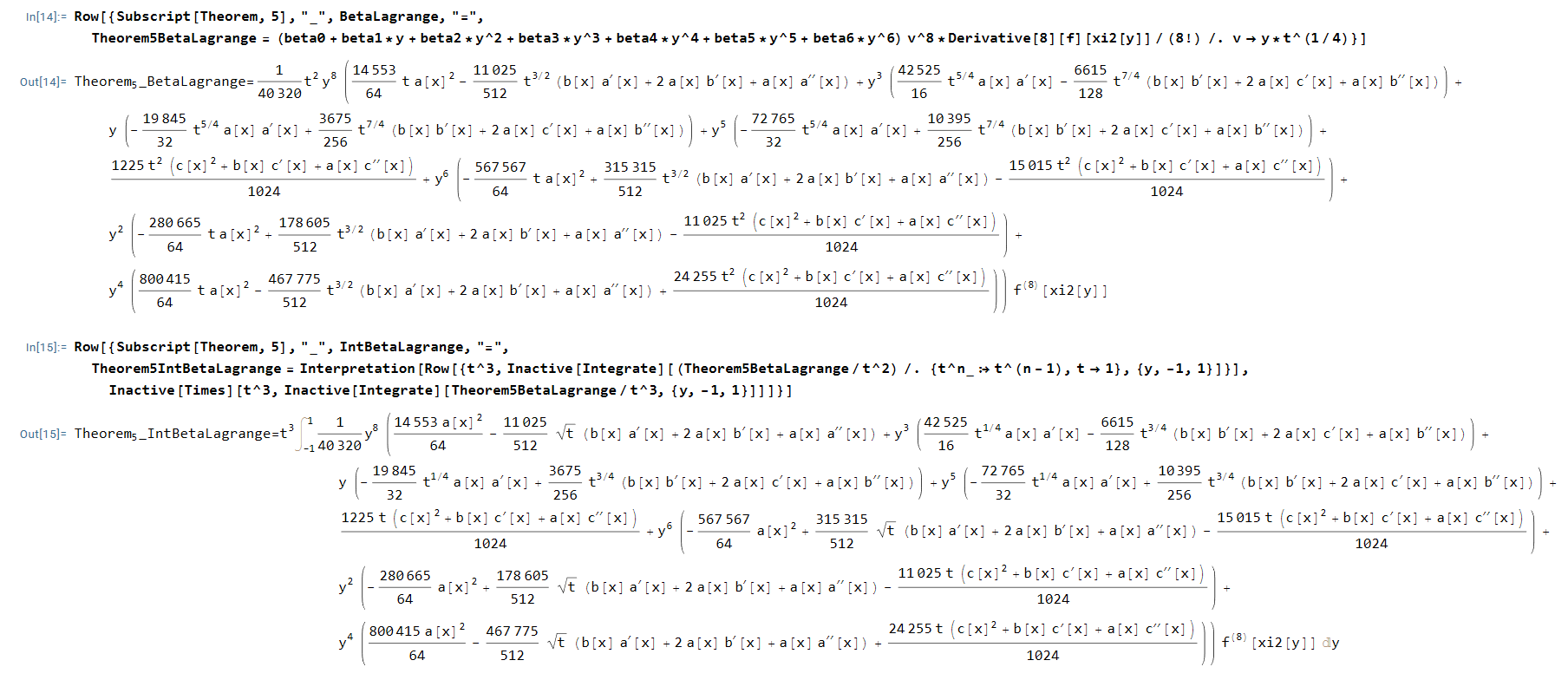}
\caption{Verification of the Chernoff function from Theorem 5 using Wolfram Mathematica 13.3. Part 4/5.}
\label{th5pic4}
\end{figure}

\newpage
Then we sum $Theorem_5\_Sum$, $Theorem_5\_IntLagrange$ and $Theorem_5\_IntBetaLagrange$. We get the final result.
\begin{figure}[H]
\includegraphics[width=18cm, height=8cm]{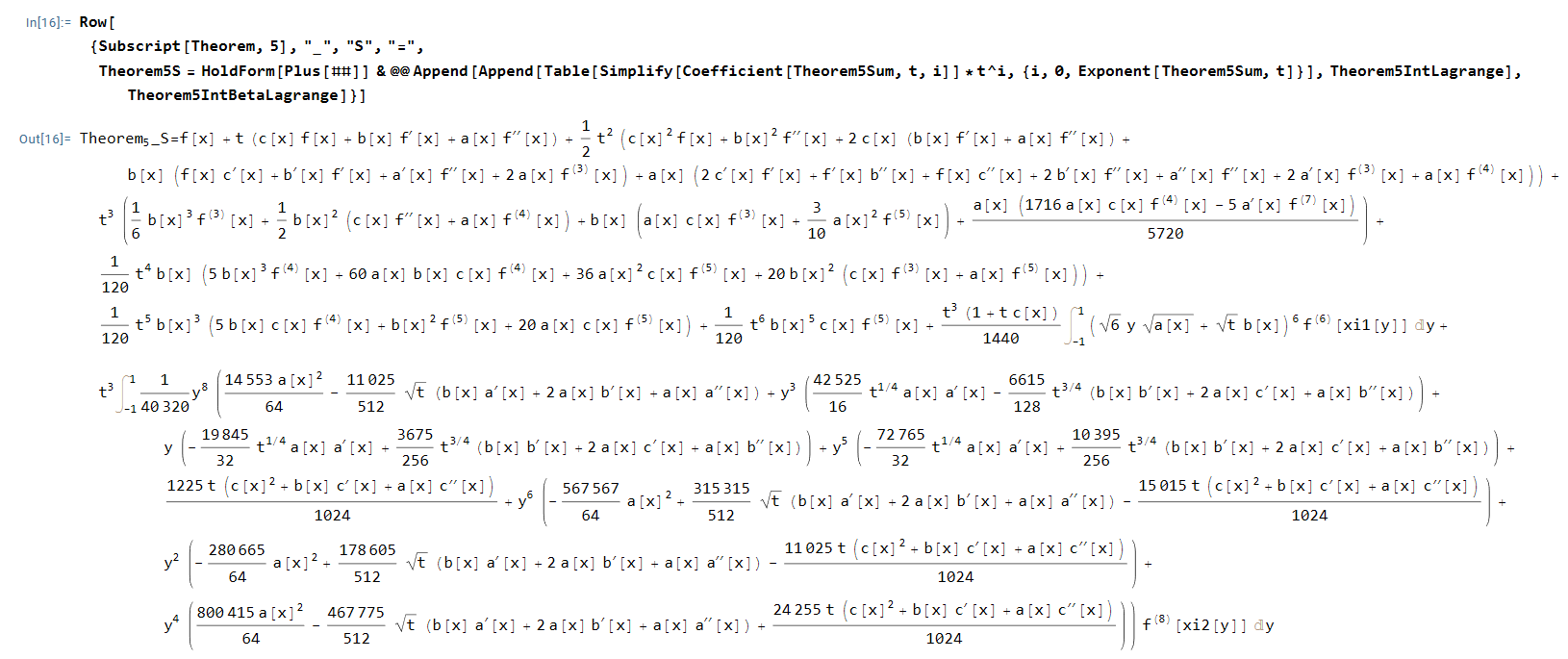}
\caption{Verification of the Chernoff function from Theorem 5 using Wolfram Mathematica 13.3. Part 5/5.}
\label{th5pic5}
\end{figure}

We have come to the following fact: the Chernoff function $S(t)$, which is defined in (\ref{Chernofffunc}) indeed satisfies item 3 of Theorem \ref{firstform}. This means that the coefficients in (\ref{Chernofffunc}) are selected in such a way that $S(t)f= f+tHf+\frac{1}{2}(tH)^2+o(t^2)$ 
as $t \to +0$ for each $f\in C_b^\infty(\mathbb{R})$.



\addcontentsline{toc}{section}{References}

\begin{thebibliography}{99}

\bibitem{Gom2019} 
C.~Batty, A.~Gomilko, Yu.~Tomilov. A Besov algebra calculus for generators of operator semigroups and related norm-estimates. // Mathematische Annalen (2019) online open access paper https://doi.org/10.1007/s00208-019-01924-2

\bibitem{BS-RFA2020} 
V.I. Bogachev, O.G. Smolyanov~O.G. Real and Functional Analysis. --- Springer. 2020. 586~p.

\bibitem{Butko2020} 
Ya.A.~Butko. The Method of Chernoff Approximation. // In: Banasiak~J., Bobrowski~A., Lachowicz~M., Tomilov~Y. (eds) Semigroups of Operators --- Theory and Applications. SOTA 2018. Springer Proceedings in Mathematics and Statistics. 2020. V.~325. P.~19-46.

\bibitem{BGS2010} Ya.A. Butko, M. Grothaus, O.G. Smolyanov. Lagrangian Feynman formulas for second-order parabolic equations in bounded and unbounded domains. // Infinite Dimensional Analyasis, Quantum Probability and Related Topics, vol. 13, No. 3 (2010), 377-392.

\bibitem{Chernoff} Paul R. Chernoff. Note on product formulas for operator semigroups. // J. Functional Analysis 2:2 (1968), 238-242. 


\bibitem{Drag2023SVMO} K. A. Dragunova, N. Nikbakht, I. D. Remizov. Numerical Study of the Rate of Convergence of Chernoff Approximations to Solutions of the Heat Equation.// Zhurnal Srednevolzhskogo matematicheskogo obshchestva. 25:4(2023), 255-272

\bibitem{KNR2026}
K.\, A.\, Katalova (Dragunova), N.\, Nikbakht, and I.\, D.\, Remizov.
\newblock {Concrete examples of the rate of convergence of Chernoff approximations: numerical results for the heat semigroup and open questions on them with full list of illustrations and Python source code.}
\newblock \emph{arXiv:2301.05284 [math.NA, math.FA]}, 2023.



\bibitem{EN2000} Engel~K.-J., Nagel~R.  One-Parameter Semigroups for Linear Evolution Equations. --- Springer-Verlag. New York. 2000. 609~p.

\bibitem{EBY}G. Evans, J. Blackledge, P. Yardley. Numerical Methods for Partial Differential Equations. --- Springer, 2000.

\bibitem{GR1}O.E.~Galkin, I.D.~Remizov. Upper and lower estimates for rate of convergence in the Chernoff product formula for semigroups of operators. // arXiv:2104.01249v2 (2021)

\bibitem{GR-IJM}O.E.~Galkin, I.D.~Remizov. Upper and lower estimates for rate of convergence in the Chernoff product formula for semigroups of operators. //  Isr. J. Math. 265, 929-943 (2025)



\bibitem{GR2}O.E.~Galkin, I.D.~Remizov. Speed of convergence of Chernoff approximations to $C_0$-semigroups of operators.// Mathematical Notes, Vol. 111, No. 2, pp. 137-139 (2022)


\bibitem{GomTom2019} A.Gomilko, S.Kosowicz, Yu.Tomilov. A general approach to approximation theory of operator semigroups// Journal de Math\'ematiques Pures et Appliqu\'ees 127 (2019) 216-267

\bibitem{Gom2014} A. Gomilko, Yu. Tomilov. On convergence rates in approximation theory for operator semigroups. //Journal of Functional Analysis 266:5 (2014), 3040-3082

\bibitem{HF} E. Hille, R.S. Phillips. Functional Analysis and Semi-groups. ---  American Mathematical Society. 1996. 819~p.

\bibitem{M2016}S. Mazumder. Numerical Methods for Partial Differential Equations. Finite Difference and Finite Volume Methods. // Academic Press, 2016

\bibitem{Zag-2} Hagen Neidhardt,
Artur Stephan and Valentin A. Zagrebnov. Remarks on the operator-norm convergence of the Trotter product
formula. // Integral Equations and Operator Theory, 90:15 (2018). 

\bibitem{Zag-3} Hagen Neidhardt, Artur Stephan, Valentin A. Zagrebnov. Operator-Norm Convergence of the Trotter Product Formula on Hilbert and Banach Spaces: A Short Survey. // Current Research in Nonlinear Analysis, pp 229-247 (2018).


\bibitem{OSS2012} Yu.N. Orlov, V.Zh. Sakbaev,  O.G. Smolyanov. Rate of convergence of Feynman approximations of semigroups generated by the oscillator Hamiltonian.// Theoretical and Mathematical Physics 172, 987–1000 (2012)

\bibitem{book2018} Daniele Antonio Di Pietro, Alexandre Ern, Luca Formaggia (Eds.). Numerical Methods for PDEs. State of the Art Techniques. --- Springer 2018.

\bibitem{Prud2020} P.S. Prudnikov. Speed of convergence of Chernoff approximations for two model examples: heat equation and transport equation// arXiv 2020

\bibitem{R-AMC} I.D. Remizov. Approximations to the solution of Cauchy problem for a linear evolution equation via the space shift operator (second-order equation example). // Applied Mathematics and Computaton 328 (2018), 243-246.

\bibitem{remizov_workshop_2018} I.  D.  Remizov.   On  estimation  of  error  in  approximations  provided  by  chernoff’s  product formula.// International  Conference  ’ShilnikovWorkshop-2018’  dedicated  to  the  memory  of outstanding  Russian  mathematician  Leonid  Pavlovich  Shilnikov  (1934-2011), book of abstracts, 38-41, 2018.

\bibitem{R-JFA2016} I.D. Remizov. Quasi-Feynman formulas -- a method of obtaining the evolution operator for the Schr\"odinger equation. // Journal of Functional Analysis 270:12 (2016), 4540-4557.

\bibitem{R-JMP2019}I.D. Remizov. Solution-giving formula to Cauchy problem for multidimensional parabolic equation with variable coefficients.// Journal of Mathematical Physics, 60:7 (2019), 071505

\bibitem{R2016} V. Ruas. Numerical Methods for Partial Differential Equations: An Introduction. --- Wiley, 2016. 300 p.

\bibitem{SmHist} O.G. Smolyanov. Feynman formulae for evolutionary equations. // Trends in Stochastic Analysis, London Mathematical Society Lecture Notes Series 353, 2009.

\bibitem{STT} O.G. Smolyanov, A.G. Tokarev, A. Truman. Hamiltonian Feynman path integrals via the Chernoff formula. // J. Math. Phys. 43, 10 (2002) 5161-5171.

\bibitem{STmzm} O. G. Smolyanov, A. Truman. Feynman Formulas for Solutions of the Schr\"odinger Equation on Compact Riemannian Manifolds. // Math. Notes, 68:5 (2000), 668-671


\bibitem{SWWcan} O.G. Smolyanov, H.v. Weizs\"acker, and O. Wittich. Brownian motion on a manifold as limit of stepwise conditioned standard Brownian motions. // Stochastic processes, Physics and Geometry: New Interplays. II: A Volume in Honour of S. Albeverio, volume 29 of Can. Math. Soc. Conf. Proc., pages 589–602. Am. Math.
Soc., 2000.

\bibitem{SWWdan} O.G. Smolyanov, H.v. Weiz\"sacker, O. Wittich. Diffusion on compact Riemannian manifolds, and surface measures. // Doklady Math. 2000. V. 61. P. 230-234.


\bibitem{VVGKR2020} A.V. Vedenin, V.S. Voevodkin, V.D. Galkin, E.Yu. Karatetskaya, I.D. Remizov. Speed of Convergence of Chernoff Approximations to Solutions of Evolution Equations// Mathematical  Notes 108(3) 451-456 (2020)

\bibitem{Zag-1} V.A.Zagrebnov. Comments on the Chernoff $\sqrt{n}$-lemma. // In book [Series of Congress Reports: Functional Analysis and Operator Theory for Quantum Physics, The Pavel Exner Anniversary Volume. Jaroslav Dittrich, Hynek Kova\v{r}\'ik, Ari Laptev (eds). European Mathematical Society, 2017] pp. 564-573, 2017.



\bibitem{Zag-JFA2020} V.A. Zagrebnov. Notes on the Chernoff product formula.//
Journal of Functional Analysis  279:7 (2020) 108696

\end{thebibliography}
\end{document}